\documentclass[twoside,a4paper,12pt]{amsart}

\usepackage{amsfonts, amsbsy, amsmath, amssymb, latexsym}
\usepackage{mathrsfs,array,enumitem}
\usepackage[top=30mm,right=30mm,bottom=30mm,left=30mm]{geometry}
\usepackage{stmaryrd}
\usepackage{bm}
\usepackage{bbding}
\usepackage{pdflscape}
\usepackage{caption,amsmath,amssymb,amsfonts,amsthm,latexsym,graphicx,multirow,color,hyperref,enumerate}
\usepackage[all]{xy}
\usepackage{xcolor}
\usepackage{mathrsfs}
\usepackage{amsmath}
\usepackage{latexsym}
\usepackage{amssymb}
\usepackage{amsfonts}
\usepackage{longtable}
\usepackage{colortbl}
\usepackage{graphics}
\usepackage{booktabs}
\usepackage{setspace}
\usepackage[normalem]{ulem}

\def\ov{\overline} 
\def\l{\langle} \def\r{\rangle} 

\def\div{\,\,\big|\,\,}  
\def\FF{\mathbb F}
\def\ZZ{\mathbb Z} \def\Z{{\rm Z}}
 
 \def\CC{{\mathcal C}}

\def\mod{{\sf mod~}}     
\def\Aut{{\sf Aut}}   \def\Out{{\sf Out}}

  \def\soc{{\sf soc}} \def\AG{{\rm AG}}
   
   \def\D{{\rm D}} \def\Q{{\rm Q}} \def\S{{\rm S}} 
\def\G{{\rm G}}  
\def\M{{\rm M}}

 \def\HS{{\rm HS}}
 
\def\N{{\bf N}}
\def\Z{{\bf Z}} 
\def\calO{{\mathcal O}}

\def\a{\alpha} \def\b{\beta}
\def\d{\delta} \def\g{\gamma}

  \def\GL{{\rm GL}} \def\PG{{\rm PG}}

\def\Sp{{\rm Sp}} 
\def\PSp{{\rm PSp}}

\def\GammaL{{\rm \Gamma L}}
\def\AGammaL{{\rm A\Gamma L}}
\def\PGammaL{{\rm P\Gamma L}}
\def\PSigmaL{{\rm P\Sigma L}}

\def\PGammaU{{\rm P\Gamma U}}

\def\A{{\rm A}}  
 \def\PSL{{\rm PSL}}  \def\PGL{{\rm PGL}}
\def\GL{{\rm GL}} \def\SL{{\rm SL}}
\def\AGL{{\rm AGL}}  \def\PSU{{\rm PSU}} \def\SU{{\rm SU}}
\def\GU{{\rm GU}} \def\GO{{\rm GO}}   \def\PGU{{\rm PGU}} \def\SU{{\rm SU}}
\def\Sz{{\rm Sz}} \def\McL{{\rm McL}}    \def\E{{\rm E}} 
 \def\D{{\rm D}} 
\def\G{{\rm G}}   \def\Co{{\rm Co}}  \def\HS{{\rm HS}}
\def\Ree{{\rm Ree}}

 \def\calM{{\mathcal M}}
 \def\soc{{\rm soc}}
\def\vs{\vskip0.05in}
\def\calM{{\mathcal M}}
\def\calP{{\mathcal P}}
\def\calB{{\mathcal B}}
\def\calD{{\mathcal D}}
\def\calI{{\mathcal I}}

\def\calH{{\mathcal H}}
\def\le{\leqslant}

\def\leq{\leqslant}
\def\geq{\geqslant}
\def\scrS{{\mathscr S}}

\def\ASL{{\rm ASL}} 
\def\0{{\bf 0}} \def\MA{{\textsc{Magma}}}

\newtheorem*{lem}{Lemma}
\newtheorem{theorem}{Theorem}[section]%
\newtheorem{lemma}[theorem]{Lemma}%
\newtheorem{proposition}[theorem]{Proposition}%
\newtheorem{definition}[theorem]{Definition}%

\newtheorem{example}[theorem]{Example}%
\newtheorem{construction}[theorem]{Construction}%
\newtheorem*{lemWK}{Lemma (Wagner, Kantor)}
\newtheorem{hypothesis}[theorem]{Hypothesis}%

\def\qed{{\hfill$\Box$\smallskip}
	\medbreak}

\begin{document}
	
\title[Locally 2-homogeneous block designs]{Locally 2-homogeneous block designs}
\thanks{The project was partially supported by the NNSF of China (11931005, 12526634).}

\author{Jianfu Chen}
\address{School of Mathematics and Computational Science\\
		Wuyi University, Jiangmen 529030,  P.R.China}
\email{chenjf@wyu.edu.cn; jmchenjianfu@126.com}
	
\author{Peice Hua}
\address{
SUSTech International Center for Mathematics\\
Southern University of Science and Technology,  Shenzhen 518055,  P.R.China}
\email{huapc@pku.edu.cn}
	
\author{Cai Heng Li}
\address{Department of Mathematics \\
SUSTech International Center for Mathematics\\
Southern University of Science and Technology,  Shenzhen 518055,  P.R.China}
\email{lich@sustech.edu.cn}
	
\author{Yanni Wu}
\address{Department of Mathematics \\
Southern University of Science and Technology,  Shenzhen 518055,  P.R.China}
\email{12031209@mail.sustech.edu.cn}
	
\begin{abstract}
This paper presents a classification of locally $2$-homogeneous designs,
extending Kantor's classification of 2-transitive symmetric designs (1985).\vskip0.07in

\vskip0.02in
		
\end{abstract}

\maketitle
\date\today
	
\section{Introduction}
A {\it balanced incomplete block design} (BIBD) is a point-block incidence geometry $ \calD=(\calP,\calB,\calI) $ where $\calP$ is a set of $ v $ points, $\calB$ is a set of $ b $ blocks, and $\calI\subseteq\calP\times\calB$ is an incidence relation such that\,:
\begin{itemize}
	\item each block is incident with exactly $k<v$ points, and \vskip0.02in
	\item any $ 2 $ distinct points are incident with exactly $\lambda>0$ common blocks.
\end{itemize}

It is known that each point of $ \calD $ is incident with the same number, denoted by $ r $, of blocks. Throughout this paper, we shall always denote the above parameters by $(v,b,r,k,\lambda)$. Depending on the context, a BIBD is also referred to as a {\em $ 2 $-$ (v,k,\lambda) $ design}, a {\em $ 2 $-design}, or simply a \textit{design}.\vs

In this paper, we only consider {\em simple designs} where distinct blocks have distinct sets of incident points, and thus $\calB\subseteq\calP^{\{k\}}$. We further assume that $ \calD $ is {\it non-trivial}, that is, $\calB\ne\calP^{\{k\}}$. A {\em flag} of $\mathcal{D}$ is an incident point-block pair. By Fisher's inequality, $ v\leqslant b $. In the case $ v=b $, we call $\calD$ a {\it symmetric design}.

An \textit{automorphism} of $ \calD $ is a permutation on points that preserves the blocks and also the incidence relation. Denote by $\Aut(\mathcal{D})$ the group of automorphisms of $\calD$. We call $ \calD $ a {\it $G$-flag-transitive} design if $ G\leqslant\Aut(\mathcal{D}) $ acts transitively on the set of flags of $ \calD $. The following is a long-standing open problem in design theory.\vs

\noindent {\bf Problem A.} Classify flag-transitive designs.\vs

It would be difficult to solve Problem A in full generality.
Many research have considered special parameter conditions. For example, flag-transitive designs with $ \lambda=1 $ or $ (\lambda,r)=1 $ have been well-characterized, respectively, in \cite{Buekenhout2vk1,KantorProplane} and \cite{Alavirlam1,Zies1988}. On the other hand, many others have considered stronger symmetry conditions. A design is said to be \textit{$ 2 $-transitive} if it has an automorphism group acting 2-transitively on the points. One classic result is Kantor's classification of 2-transitive symmetric designs \cite{Kantor2tr} in 1985. Later, designs with automorphism groups being primitive of rank 3 on the points were studied in \cite{rank3as,rank3ha}.\vs

Our interest in this problem stems from the following observation\,: there is a close connection between local symmetry and global symmetry of designs. 
In the following, denote by $ \a $, $ \b $, respectively, a point and a block of $ \calD $, and by $\calD(\a)$ the set of blocks incident with $\a$, by $\calD(\b)$ the set of points incident with $\b$. By the simplicity assumption, $ \b $ is identified with $ \calD(\b) $. 
	
\begin{definition}
{\rm
Let $\calD=(\calP,\calB,\calI)$ be a design and let $G\le\Aut(\calD)$. Then $\calD$ is said to be {\em $G$-locally transitive}, if for each element $\g\in\calP\cup\calB$, the stabilizer $G_\g$ is transitive on the set $\calD(\g)$. In a similar manner, one defines {\em $ G $-locally primitive, $ G $-locally $2$-homogeneous, and $ G $-locally $2$-transitive designs}. Furthermore, when such a group $G$ exists but is not specified, we omit the notation $G$.}
\end{definition}
One can easily show that a design is locally-transitive if and only if it is flag-transitive. Further, a locally 2-transitive design is point 2-transitive (Lemma \ref{loc-2-trans}), and, conversely, a point 2-transitive symmetric design (with exactly one exception) is locally 2-transitive (Proposition \ref{sym}).

Our main result, as shown below, extends Kantor's classification of 2-transitive symmetric designs to a classification of locally 2-transitive designs. Indeed, we do something slightly stronger. We classify locally 2-homogeneous designs. Unsurprisingly, all the designs are known, while most of them have a clear geometric background. The infinite families {\bf(1.1)}\,-\,{\bf(1.5)} are constructed in Section \ref{sec-exam}. Among them, the first two families are symmetric. Meanwhile, the sporadic cases in the tables are verified by $ \MA $. We note that this work continues our previous work on locally primitive designs \cite{locallyprimitive}. 

\begin{theorem}\label{thm:2-trans-Design}
Let $\calD=(\calP,\calB,\calI)$ be a $ G $-locally $ 2 $-homogeneous design. Then $ G $ is $2$-homogeneous on $\calP$, and either {\bf(1)} or {\bf(2)} occurs.\vs

\noindent {\bf(1)} $G$ is locally $2$-transitive on $\calD$ and $2$-transitive on $\calP$. One of the following holds{\rm\,:}
\begin{itemize}
\item[\bf(1.1)] $\calD=\PG(d-1,q)$ or $\ov{\PG(d-1,q)}$, and $\PSL_d(q)\lhd G \leqslant \PGammaL_d(q)${\rm\,;}
\vskip0.03in
\item[\bf(1.2)] $\calD=\scrS^{\pm}(2m)$, and $G=2^{2m}{:}\Sp_{2m}(2)${\rm\,;}
\vskip0.03in
\item[\bf(1.3)] $\calD=\AG_1(d,q)$ or $\AG_{d-1}(d,q)$, and $\ASL_d(q) \lhd G \leqslant \AGammaL_d(q)${\rm\,;}
\vskip0.03in
\item[\bf(1.4)] $\calD=\PG_1(d-1,q)$ with $d>3$, and $\PSL_d(q) \lhd G \leqslant \PGammaL_d(q)${\rm\,;}
\vskip0.03in
\item[\bf(1.5)] $\calD=U_H(q)$, $\PSU_3(q) \lhd G\leqslant \PGammaU_3(q)$ and $ \PGU_3(q)\leq G\l\phi^2\r $, where $\phi$ is the field automorphism{\rm\,;}
$  $\vskip0.03in
\item[\bf(1.6)] $ \calD $ and $ (G,G_\a,G_{\a\b},G_\b) $ are listed in {\rm Tables~\ref{table_sym}, \ref{table_nsym}}, with $(\a,\b)\in\calP\times\calB$ a flag.
\end{itemize}\vs

\noindent {\bf(2)} $G$ is not locally $2$-transitive on $\calD$. In this case, $ \calD $ and $ (G,G_\a,G_{\a\b},G_\b) $ are listed in {\rm Table~\ref{table_not2trans}}, with $(\a,\b)\in\calP\times\calB$ a flag.

\begin{table}[ht]
\tiny
\renewcommand\arraystretch{1.5}
\caption{\small Sporadic $ G $-locally 2-transitive designs (symmetric)}\label{table_sym}
\centering
\begin{tabular}{ccccccc}
\toprule[1pt]
$v$ & $k$ & $\lambda$ & $G$ & $G_{\a} \cong G_{\b}$ & $G_{\a \b}$ & {\rm Remark} \\
\midrule[1pt]

$11$ & $5$ & $2$ & $\PSL_2(11)$ & $\A_5$ & $\A_4$  \\
& $6$ & $3$ & &  & $\D_{10}$  \\ 
\hline

$15$ & $7$ & $3$ & $\A_7$ & $\PSL_2(7)$ & $\S_4$ & $\calD=\PG(3,2)$\\
& $8$ & $4$ & & & $7{:}3$ & $\calD=\overline{\PG(3,2)}$ \\ 
\hline

$16$ & $6$ & $2$ & $2^4{:}\A_6$ & $\A_6$ & $\A_5$ & $\calD=\scrS^{-}(4)$ \\
& $10$ & $6$ & & & $3^2{:}4$ & $\calD=\scrS^{+}(4)$\\ 
\hline
		
$64$ & $28$ & $12$ & $2^6{:}\PSU_3(3).\calO$ & $\PSU_3(3).\calO$ & $3_+^{1+2}{:}8.\calO$ & $\calO=1,2\,;\calD=\scrS^{-}(6)$\\ 
\hline
		
$176$ & $126$ & $90$ & $\HS$ & $\PSU_3(5){:2}$ & $5_+^{1+2}{:}8{:}2$  \\ 
\bottomrule[1pt]
\end{tabular}
\end{table}

\begin{table}[ht]
\tiny
\renewcommand\arraystretch{1.5}
\caption{\small Sporadic $ G $-locally 2-transitive designs (non-symmetric)}\label{table_nsym}
\centering
\begin{tabular}{cccccccccc}
\toprule[1pt]
$v$ & $b$ & $r$ & $k$ & $\lambda$ & $G$ & $G_{\a}$ & $G_{\a \b}$ & $G_{\b}$ & {\rm Remark} \\
\midrule[1pt]

$12$ & $22$ & $11$ & $6$ & $5$ & $\M_{11}$ & $\PSL_2(11)$ & $\A_5$ & $\A_6$ \\

$15$ & $35$ & $7$ & $3$ & $1$ & $\A_7$ & $\PSL_3(2)$ & $\S_4$ & $(\A_4 \times 3){:}2$ & $\calD=\PG_1(3,2)$\\

$16$ & $20$ & $5$ & $4$ & $1$ & $\AGammaL_1(16)$ & $\GammaL_1(16)$ & $3{:}4$ & $2^2{:}(3{:}4)$ & $\calD=\AG_1(2,4)$\\
$16$ & $30$ & $15$ & $8$ & $7$ & $2^4{:}\A_7$ & $\A_7$ & $\PSL_3(2)$ & $2^3{:}\PSL_3(2)$ & $\calD=\AG_3(4,2)$\\

$21$ & $56$ & $16$ & $6$ & $4$ & $\PSL_3(4).\calO$ & $2^4{:}\A_5.\calO$ & $\A_5.\calO$ & $\A_6.\calO$ & $\calO=1, 2$\\

$22$ & $77$ & $21$ & $6$ & $5$ & $\M_{22}.\calO$ & $\PSL_3(4).\calO$ & $2^4{:}\A_5.\calO$ & $2^4{:}\A_6.\calO$ & $\calO=1, 2$\\
\bottomrule[1pt]
\end{tabular}
\end{table}

\begin{table}[ht]
\tiny
\renewcommand\arraystretch{1.5}
\caption{\small $ G $-locally $2$-homogeneous but not $ G $-locally $2$-transitive designs}\label{table_not2trans}
\centering
\begin{tabular}{cccccccccc}
\toprule[1pt]
$v$ & $b$ & $r$ & $k$ & $\lambda$ & $G$ & $G_{\a}$ & $G_{\a \b}$ & $G_{\b}$ & {\rm Remark} \\
\midrule[1pt]

$7$ & $7$ & $3$ & $3$ & $1$ & $7{:}3$ & $3$ & $1$ & $3$ & $\calD=\PG(2,2)$\\

$8$ & $14$ & $7$ & $4$ & $3$ & $\AGammaL_1(8)$ & $\GammaL_1(8)$ & $3$ & $\A_4$ & $\calD=\AG_2(3,2)$\\

$8$ & $14$ & $7$ & $4$ & $3$ & $\PSL_2(7)$ & $7{:}3$ & $3$ & $\A_4$ &  $\calD=\AG_2(3,2)$\\
\bottomrule[1pt]
\end{tabular}
\end{table}
\end{theorem}

\noindent {\bf Remark on Tables \ref{table_sym}\,-\,\ref{table_not2trans}.} In these tables, several designs, as indicated, have already appeared in part {\bf(1)}. The tables indeed provide new groups $ G $ being locally $ 2 $-transitive/$ 2 $-homogeneous on them. In addition, a few designs are in fact \textit{$ 3 $-designs} (defined in Section \ref{sec-preli}). As a consequence, one obtains new designs by deleting a fixed point $ \a $ as well as the blocks in $ \calD(\a) $ from them. More explicitly, one obtains $ \overline{\PG(3,2)} $ from $ \AG_3(4,2) $, and obtains the $ 2 $-$ (11,6,3) $-design or the design on $ 21 $ points in Table \ref{table_nsym}, respectively, from the design on $ 12 $ or $ 22 $ points in Table \ref{table_nsym}.\vs

It should be noted that this result relies on the well-known classification of finite 2-transitive groups (see, for example, \cite[Thm.\,5.3]{as2trans}). It follows that a 2-transitive group is of either almost simple or affine type. Further, finite 2-homogeneous groups are also known\,; refer to Wagner \cite[Thm\,2]{Wagner1965} and Kantor \cite[Thm.\,1]{Kantor1972}.

\begin{lemWK}\label{2-homo}
Let $G$ be a finite $2$-homogeneous permutation group on $ \Omega $. Then either $ G $ is $2$-transitive on $ \Omega $, or $G\leqslant \AGammaL_1(q)$ with $|\Omega|=q\equiv3\,(\mod 4)$.
\end{lemWK}

This paper is organized as follows. By Lemma \ref{loc-2-trans}, letting $ \calD $ be a $ G $-locally $2$-homogeneous design, then $ G $ is 2-homogeneous on the set $ \calP $ of points. Thus, by the above lemma, we have the following three cases to consider\,:
\begin{itemize}
\item[\bf(A)] $G$ is $2$-transitive on $ \calP $ and $ \calD $ is symmetric.\vs
\item[\bf(B)] $G$ is $2$-homogeneous on $ \calP $ of affine type, and either $G\leqslant \AGammaL_1(v)$\,; or $G\not\leqslant \AGammaL_1(v)$, $ \calD $ is non-symmetric.\vs
\item[\bf(C)] $G$ is $2$-transitive on $ \calP $ of almost simple type and $ \calD $ is non-symmetric.
\end{itemize}
\noindent We will discuss {\bf(A)}\,-\,{\bf(C)}, respectively, in Sections \ref{sec-Kantor}\,-\,\ref{sec-as}. Before that, we give some properties of design $ \calD $ in Section \ref{sec-preli} and construct the infinite families in Section \ref{sec-exam}.\vs

At the end of this section, we introduce the following notation for groups, which will be used throughout this paper\,:

- $\ZZ_n$\,: the cyclic group of order $n$\,;\vskip0.01in 
- $\E_{p^n}$ or $q^d$\,: an elementary abelian group of order $p^n$ or $q^d$, where $ q $ is a $ p $-power\,;\vskip0.01in 
- $[n]$\,: a group of order $n$ of unspecified structure\,;\vskip0.01in 
- $A^{m+n}$\,: a group with a normal subgroup $A^m$ such that the quotient is isomorphic to $A^n$, where $A$ is an elementary abelian group.

\section{Locally 2-homogeneous designs}\label{sec-preli}
In this section we give some properties of locally 2-homogeneous designs, as well as a generic construction of them. The first result is basic but commonly used.

\begin{lem}
The following restrictions on the design (not necessarily simple) parameters hold{\rm\,:} {\rm (1)} $ vr=bk ${\rm\,; (2)} $ \lambda(v-1)=r(k-1) ${\rm\,; (3)} $ v\leqslant b $, $ k\leq r${\rm \,; (4)} $\lambda<r $, $ \lambda v<rk $.
\end{lem}  

The next observation indicates that local symmetry implies symmetry on points.

\begin{lemma}\label{loc-2-trans}
Let $\calD=(\calP,\calB,\calI)$ be a $ G $-locally transitive design. Then $\calD$ is $G$-flag transitive.
If the stabilizer $G_\b$ is $2$-homogeneous on $\calD(\b)$ for each $ \b\in\calB $, then $G$ is $2$-homogeneous on $\calP$.
Further, if $G_\b$ is $2$-transitive on $\calD(\b)$ for each $ \b\in\calB $, then $G$ is $2$-transitive on $\calP$.

\end{lemma}

\begin{proof}

Since $G$ is locally transitive on $\mathcal{D}$, for any two flags $(\a,\b)$ and $(\a',\b')$ on $\calD$, we can find a block $\b''\in \calD(\a) \cap \calD(\a')$, such that the flags $(\a,\b),(\a,\b''),(\a',\b''),(\a',\b')$ lie in the same $G$-orbit. Hence $ G $ is flag-transitive on $\calD$, and $G$ is transitive on $\calB$.

For any two 2-subsets of $\calP$, say $\{\a_1,\a_2\}$, $\{\a_1',\a_2'\}$, let $\b,\b'\in\calB$ be two blocks such that $\a_1,\a_2\in\calD(\b)$ and $\a_1',\a_2'\in\calD(\b')$.
There exists some $g\in G$ mapping $ \b $ to $\b'$, so that, $\a_1^g,\a_2^g\in\calD(\b')$. 
Since $G_\b$ is $2$-homogeneous on $\calD(\b)$, there exists some $ h\in G_\b $ mapping $\{\a_1,\a_2\}^{g}$ to $ \{\a_1',\a_2'\}$. It follows that $G$ is 2-homogeneous on $\calP$. By a similar argument, if $G_\b$ is $2$-transitive on $\calD(\b)$, then $G$ is $2$-transitive on $\calP$.
\end{proof}

We then present a generic construction of locally 2-homogeneous 2-designs.

\begin{construction}\label{cons}
{\rm Let $G$ be a group which has two subgroups $H,K$ such that $G \ne HK$, $G$ is faithful and 2-homogeneous on $[G:H]$, $H$ is $ 2 $-homogeneous on $[H:H\cap K]$, and meanwhile, $K$ is $ 2 $-homogeneous on $[K:H\cap K]$. Let $\calP=[G:H]$, $\calB=[G:K]$. Define the incidence relation $\calI$ as \[Hx\,\calI\, Ky\,\Leftrightarrow\, Hx\cap Ky\not=\emptyset, \mbox{\ for\ } Hx\in\calP,\,  Ky\in\calB.\]}
\end{construction}

\begin{lemma}\label{lem:cons}
The geometry $\calD=(\calP,\calB,\calI)$ defined above is a simple $G$-locally $2$-homogeneous design, with $ v=|G:H| $, $ b=|G:K| $, $ r=|H:H\cap K| $ and $ k=|K:H\cap K| $.
\end{lemma}

\begin{proof}
With the right multiplication action of $ G $ on $ \calP $ and $ \calB $ we have $G\leq\Aut(\calD) $. 
Let $\a:=H\in\calP$ and $\b:=K\in\calB$. Since $Hx\cap K\not=\emptyset$ if and only if $x \in HK$, we have $K \leqslant G_{\calD(\b)} \subseteq HK$. Then $HK=HG_{\calD(\b)}$, and $G\ne HK$ implies that $v\ne k$.
As $ 2 $-homogeneneity implies primitivity, it follows that $H\cap K$ is maximal on $H$.
Note that $H\cap K\leq H\cap G_{\calD(\b)}\leq H$. 
If $ H\cap G_{\calD(\b)}= H$, then $H \leqslant G_{\calD(\b)}$.
Since $H$ is maximal on $G$ and $G_{\calD(\b)} \ne G$, we have $H= G_{\calD(\b)}$ and $K \leqslant H$, a contradiction. Hence $H\cap K= H\cap G_{\calD(\b)}$,
and
$$\frac{|H||G_{\calD(\b)}|}{|H\cap G_{\calD(\b)}|}=|HG_{\calD(\b)}|=|HK|=\frac{|H||K|}{|H\cap K|}.$$
Thus $K=G_{\calD(\b)}$. Assume that $\calD$ is not simple, and $\calD(\b)=\calD(\b')$, where $K=\b\ne \b'$.
Since $G$ is transitive on $\calB$, there exists $g\in G$ such that $\b'=\b^g=Kg$, and $K\ne Kg$. Thus $g\notin K$.
Note that $\calD(\b)^g=\calD(\b^g)=\calD(\b')=\calD(\b)$. Hence $g\in G_{\calD(\b)}$, implying that $K<G_{\calD(\b)}$. Hence $\calD$ is simple.

Since $ G $ is $ 2 $-homogeneous on $\calP$, each pair of points is incident with the same number, say $\lambda$, of blocks. Fix a point $H\in\calP$ and take an element $x\in K \setminus H\cap K$. Then $H$, $Hx$ are two distinct points incident with $K$. Thus, $\lambda>0$ and $\calD$ is a design. The action of $ G_H$ on $ \calD(H) $ is equivalent to the action of $ H $ on $[H:H\cap K]$, so, by assumption, it is 2-homogeneous. Similarly, $G_K$ is 2-homogeneous on $\calD(K)$. Hence $G$ is locally $ 2 $-homogeneous on $ \calD $. Further, the parameters are clear.
\end{proof}

According to Construction \ref{cons}, for a given 2-homogeneous permutation group $G$ on a set $\calP$ of degree $v$ and a parameter set $(v, b, r, k,\lambda)$, we may apply the following procedure to verify whether there exists a $G$-locally 2-homogeneous 2-design with these parameters.

Step 1. Find conjugacy classes of subgroups of index $b$ in $G$.

Step 2. Take a representative $K$ from such a conjugacy class, as a candidate for a block stabilizer.

Step 3. Find a point-orbit $O$ of $K$ of length $k$, as a candidate for a block.

Step 4. Let $H$ be the stabilizer of $G$ of a point in $O$. Verify whether $|H:H\cap K|=r$ holds.

Step 5. Check whether the two coset actions, $H$ on $[H:H\cap K]$, and $K$ on $[K:H\cap K]$, are 2-homogeneous.

This procedure can be implemented using {\sc Magma}, and will be applied in the proof of the main theorems for some sporadic cases.

\vs\vs\vs

The following result gives a key property of locally 2-homogeneous designs. A design is called \textit{quasi-symmetric} if there are {\it intersection numbers} $ \ell_1 $, $ \ell_2 $ such that any two distinct blocks have exactly $ \ell_1 $ or $ \ell_2 $ points in common. 

\begin{lemma}\label{c}
Let $\calD=(\calP,\calB,\calI)$ be a $ G $-locally transitive design. If the point stabilizer $G_{\a}$ is $2$-homogeneous on $\calD(\a)$, then $ \calD $ is either symmetric, or quasi-symmetric with intersection numbers $0$, $c$, where $c=\frac{(k-1)(\lambda-1)}{r-1}+1$ is an integer determined by the parameters.
\end{lemma}

\begin{proof}
Note that there exist two distinct blocks intersecting in $ c>0 $ points. Define $$S=\{ \{\b,\gamma\}\in \calB^{\{2\}} \ |\ \calD(\b) \cap \calD(\gamma) \ne \emptyset \}.$$ 
For any $\{\b,\gamma\}$, $\{\b',\gamma'\} \in S$, there exist $\alpha \in \calD(\b) \cap \calD(\gamma)$, $\alpha' \in \calD(\b') \cap \calD(\gamma')$, or equivalently, $\b, \gamma \in \calD(\a)$, $\b', \gamma' \in \calD(\a')$. By Lemma \ref{loc-2-trans}, $G$ is flag-transitive on $ \calD $, thus transitive on $ \calP $\,; meanwhile, $G_{\a}$ is $2$-homogeneous on $\calD(\a)$. It follows that $ G $ is transitive on $S$. Thus, either $ \calD $ is quasi-symmetric with intersection numbers $ 0 $, $ c $\,; or any two distinct blocks of $ \calD $ intersect in $ c $ points, so that, the dual of $ \calD $ is also a $ 2 $-design, and hence $ v=b $, $ \calD $ is symmetric. 

We then determine the number $ c $. If $\lambda=1$, no two blocks intersect in more than one point, so $ c=1 $, as required. Suppose $\lambda\geqslant 2$. Then there are two blocks intersecting in at least two points, so $ c\geqslant 2 $. Let $ \a\in\calP $ be a fixed point, and define a new geometry $ \calD'=(\calP',\calB',\calI') $ where $\calP'=\calD(\a)$, $ \calB'=\calP \setminus \{\a\}$, and $ \calI' $ inherits from $ \calI $. For any two blocks $\b,\gamma\in\calD(\a)$, with the exception of $ \a $, they intersect in $ c-1 $ points. Thus, $\calD'$ is a $2$-design (not necessarily simple) with parameters $$(v',b',r',k',\lambda')=(r,\,v-1,\,k-1,\,\lambda,\,c-1).$$ It then follows from $\lambda'(v'-1)=r'(k'-1)$ that $c=\frac{(k-1)(\lambda-1)}{r-1}+1.$
\end{proof}

\begin{lemma}\label{k<=v/2}
Let $\calD=(\calP,\calB,\calI)$ be a (non-trivial) $ G $-locally $2$-homogeneous design. Then $ 3\leq k \leq r $ and $k\ne v-2$, $v-1$. Further, if $ r=3 $, then $ \calD $ is the unique symmetric design $ \PG(2,2) ${\rm\,;} if $ \calD $ is non-symmetric, then $k \leqslant v/2$.
\end{lemma}

\begin{proof}
By Lemma \ref{loc-2-trans}, $ G $ is $ 2 $-homogeneous on $ \calP $, so it is also $ i $-homogeneous on $ \calP $ for $ i\in\{1,2,v-2,v-1\} $. It follows from $\calB\ne\calP^{\{k\}}$ that $ k\ne i $. Further, if $ r=3 $, then $ k=3 $, and the only solution for $ \lambda(v-1)=r(k-1) $ is $ (v,k,\lambda)=(7,3,1) $. It is known that the unique $ 2 $-$(7,3,1) $ design is $ \PG(2,2) $. At last, if $ \calD $ is non-symmetric, by lemma \ref{c} there exist two disjoint blocks, so $k\leq v/2$.
\end{proof}

The following are some properties of quasi-symmetric designs. A $ 2 $-design is called a \textit{$ t $-$ (v,k,\lambda_t) $ design}, if any $ t>2 $ distinct points are incident with exactly $\lambda_t$ common blocks. It is easy to check that a $t$-$(v,k,\lambda_t)$ design is an $s$-$(v,k,\lambda_s)$ design for $2\leq s\leqslant t$, where $\lambda_s=\lambda_t {v-s \choose t-s}/{k-s \choose t-s}$.

\begin{lemma}\cite[Prop.\,3.6]{quasi-symmetric}\label{quasi_sym_4design}
Let $\calD$ be a $2$-$ (v,k,\lambda) $ design with $4\leqslant k \leqslant v-4$. Then any two of the following imply the third{\rm\,:}
\begin{itemize}
\item [\rm (1)] $\calD$ is quasi-symmetric{\rm\,;}\vskip0.02in
\item [\rm (2)] $\calD$ is a $4$-design{\rm\,;}\vskip0.02in
\item [\rm (3)]	$b =\frac{1}{2}v(v-1)$.	
\end{itemize}
\end{lemma}

\begin{lemma} \cite[Thm.\,3.7]{quasi-symmetric} \label{quasi_sym_b=2v-2}
Let $ \calD $ be a quasi-symmetric design with $b= 2v-2$. Then $ \calD $ is either a Hadamard $3$-design\,\footnote{\,For the construction of Hadamard matrix and Hadamard 3-design, see \cite[Exam.\,19.3]{Course}.} or the unique $2$-$(6,3,2)$ design.
\end{lemma}

The next lemma is a simple modification of \cite[Lem.\,8.1]{locallyprimitive}.

\begin{lemma}\label{socflagtrans}
Let $\mathcal{D}=(\calP,\calB,\calI)$ be a $G$-locally primitive design with $G$ an almost simple group. Let $ \soc(G)=T $. Then $ T $ is flag-transitive on $\calD$.
\end{lemma}

\begin{proof}
By \cite[Thm.\,1.3]{locallyprimitive}, $G$ is primitive on $ \calP $ and quasiprimitive on $\calB$, so $T$ is transitive on both $\calP$ and $\calB$. Let $ (\alpha,\b) $ be a flag. Suppose that $T_\a$ fixes all elements of $\calD(\a)$ and $T_\b$ fixes all elements of $\calD(\b)$. Then $T_\b\le T_\a$ and $T_\a\le T_\b$, so that, $T_\a=T_\b$. It follows that $T_\a$ fixes all elements of $\calP\cup\calB$, and so $T_\a=1$. This is impossible since $T$ is not regular on $\calP$. Thus, either $T_\a^{\calD(\a)}\not=1$ or $T_\b^{\calD(\b)}\not=1$. Since $G_\a^{\calD(\a)}$ and $G_\b^{\calD(\b)}$ are primitive permutation groups, it follows that either $T_\a^{\calD(\a)}$ or $T_\b^{\calD(\b)}$ is transitive, and so $T$ is flag-transitive on $\calD$.
\end{proof}

\section{Infinite families}\label{sec-exam}
In this section we construct the infinite families of locally 2-homogeneous designs.

\begin{example}\label{ex:Pro-G} 
{\rm 
The following designs come from the linear space $V=\FF_q^d$, $ d\geq 3 $. Let $G=\SL_d(q)$ and $\calM_{i,j}$ be the set of $i \times j$-matrices over $\FF_q$. For a $ k $-subspace $ W<V $, the stabilizer of $ W $ in $ G $ is known as a parabolic subgroup $P_k$ of form 
\[{\renewcommand{\arraystretch}{1.4}
\begin{array}{lll}
P_k&=&\left<
\left[\begin{array}{cc}
A&0\\ C&D
\end{array}\right]
\Bigg|\ A\in\GL_k(q),\ D\in\GL_{d-k}(q),\ |AD|=1,\ C\in \calM_{k,d-k} \right>\\
&\cong&[q^{k(d-k)}]{:}((\SL_k(q)\times\SL_{d-k}(q)).(q-1)),
\end{array}}\]
	
\noindent {\bf (1)} {\bf The projective space} $\PG(d-1,q)=(\calP,\calH,\calI)$ is defined as follows\,:
\begin{itemize}
\item[] $\calP$\,: the set of $ 1 $-subspaces of $V$,\vs
\item[] $\calH$\,: the set of hyperplanes of $V$,\vs
\item[] $\a\,\calI\,\b \Leftrightarrow \a\subset\b$, for $ \a\in\calP $, $ \b\in\calH $.
\end{itemize}

Let $ \calD=\PG(d-1,q) $. It is known that $ \calD $ is a symmetric $ 2 $-$ (\frac{q^{d}-1}{q-1}, \frac{q^{d-1}-1}{q-1}, \frac{q^{d-2}-1}{q-1}) $ design with $ \Aut(\calD)=\PGammaL_d(q) $. Note that $ G^\calP\cong\PSL_d(q)\leqslant\Aut(\calD) $, and further, $ G^\calP $ is 2-transitive on both $ \calP $ and $ \calH $. Let $(\a,\b)\in \calP\times\calH$ be a flag. Then $$G_{\a}\cong G_{\b} \cong q^{d-1}{:}\GL_{d-1}(q),$$ and	$ G_{\a\b}=G_\a\cap G_\b $ is of form \begin{align*}
G_{\a\b}&=\left<
\left[\begin{array}{ccc}
A&0&0\\
B&C&0\\
D&E&F
\end{array}\right]
\Bigg|\,A,F\in\GL_1(q),\,C\in\GL_{d-2}(q),\, B,E^{\sf T}\in\calM_{d-2,1},\,D\in \calM_{1,1}\right>\\
&\cong q^{d-1}{:}((q^{d-2}{:}(\GL_1(q)\times\GL_{d-2}(q)))), \mbox{\ where $ |A|\cdot|C|\cdot |F|=1 $.}
\end{align*}
\noindent Thus, $|\calD(\a)|=|G_\a:G_{\a\b}|=\frac{q^{d-1}-1}{q-1} $, $|\calD(\b)|=|G_\b:G_{\a\b}|=\frac{q^{d-1}-1}{q-1} $, and so $ G $ is flag-transitive (locally transitive) on $ \calD $. Further, $G_\b^{\calD(\b)}=\PGL(\b)\cong\PGL_{d-1}(q)$ is 2-transitive on $ \calD(\b) $. Similarly, $G_\a^{\calD(\a)}\cong\PGL_{d-1}(q)$ is $ 2 $-transitive on $\calD(\a)$. It follows that $ \calD $ is $ K $-locally $ 2 $-transitive for each $ K $ with $ \PSL_d(q)\lhd K\leqslant \PGammaL_d(q) $. \vs\vs

\noindent {\bf(2)} {\bf The complement}  $\overline{\PG(d-1,q)}=(\calP,\calH,\calI')$ is defined as follows\,:
\begin{itemize}
\item[] $\calP$, $\calH$\,: as in {\bf(1)}, and $\a\,\calI'\,\b \Leftrightarrow \a\not\subset\b$, for $ \a\in\calP $, $ \b\in\calH $. \vskip1pt
\end{itemize}

Let $ \calD=\overline{\PG(d-1,q)} $. It is known that $ \calD $ a symmetric $ 2 $-$ (\frac{q^{d}-1}{q-1},q^{d-1}, (q-1)q^{d-2}) $ design with $ \Aut(\calD)=\PGammaL_d(q) $. Note that $ G^\calP\cong\PSL_d(q)\leqslant\Aut(\calD) $, and further, $ G^\calP $ is 2-transitive on both $ \calP $ and $ \calH $. Let $(\a,\b)\in \calP\times\calH$ be a flag. Then $$G_{\a}\cong G_{\b} \cong q^{d-1}{:}\GL_{d-1}(q),$$ and $ G_{\a\b}=G_\a\cap G_\b $ is of form
\begin{align*}
G_{\a\b}&=\left<
\left[\begin{array}{cc}
A&0\\0&x
\end{array}\right]
\Bigg|\ A\in\GL_{d-1}(q),\ x=|A|^{-1}
\right>\cong\GL_{d-1}(q).
\end{align*}
\noindent Thus, $|\calD(\a)|=|G_\a:G_{\a\b}|=q^{d-1} $, $|\calD(\b)|=|G_\b:G_{\a\b}|=q^{d-1}  $, and so $ G $ is flag-transitive (locally transitive) on $ \calD $. Further, $G_\b^{\calD(\b)}\cong\AGL(\b)\cong q^{d-1}{:}\GL_{d-1}(q) $ is 2-transitive on $ \calD(\b) $. Similarly, $G_\a^{\calD(\a)}\cong q^{d-1}{:}\GL_{d-1}(q)$ is $ 2 $-transitive on $\calD(\a)$. It follows that $ \calD $ is $ K $-locally $ 2 $-transitive for $ \PSL_d(q)\lhd K\leqslant \PGammaL_d(q) $. \vs\vs
		
\noindent {\bf (3)} {\bf The design} $\PG_1(d-1,q)=(\calP,\calB,\calI)$, $ d>3 $, is defined as follows\,:\begin{itemize}
\item[] $\calP$\,: the set of $ 1 $-subspaces of $V$,\vs
\item[] $\calB$\,: the set of $ 2 $-subspaces of $V$,\vs
\item[] $\a\,\calI\,\b \Leftrightarrow \a\subset\b$, for $ \a\in\calP $, $ \b\in\calB $.
\end{itemize}

Let $ \calD=\PG_1(d-1,q) $. It is easy to check that\,: $ \calD $ is a non-symmetric $ 2 $-design with parameters $(v,b,r,k,\lambda)=(\frac{q^d-1}{q-1},\frac{(q^d-1)(q^{d-1}-1)}{(q^2-1)(q-1)},\frac{q^{d-1}-1}{q-1},\frac{q^{2}-1}{q-1},1)$, and $ \Aut(\calD)=\PGammaL_d(q) $. Note that $ G^\calP\cong\PSL_d(q)\leqslant\Aut(\calD) $, and further, $ G^\calP $ is 2-transitive on $ \calP $ and transitive on $ \calB $. Let $(\a,\b)\in\calP\times\calB$ be a flag. 
Then
\[G_\a=q^{d-1}{:}\GL_{d-1}(q), \ G_\b=[q^{2(d-2)}]{:}((\SL_2(q)\times\SL_{d-2}(q)).(q-1)),\] and $ G_{\a\b}=G_\a\cap G_\b $ is of form 
\begin{align*}
G_{\a\b}&=\left<
\left[\begin{array}{ccc}
A&0&0\\ B&C&0\\ D&E&F
\end{array}\right]
\Bigg| \ A,C\in\GL_1(q), \ B\in\calM_{1,1}, \ D,E\in\calM_{d-2,1},\ F\in\GL_{d-2}(q)
\right>\\
&\cong{q}^{d-1}{:}(q^{d-2}{:}(\GL_1(q)\times\GL_{d-2}(q))),\ \mbox{where $ |A|\cdot |C|\cdot|F|=1 $.}
\end{align*}}

\noindent Thus, $ |{\calD(\a)}|=|G_\a:G_{\a\b}|=\frac{q^{d-1}-1}{q-1} $, $|\calD(\b)|=|G_\b:G_{\a\b}|=\frac{q^{2}-1}{q-1} $, and so $ G $ is flag-transitive (locally transitive) on $ \calD $. Further, $G_\a^{\calD(\a)}\cong\PGL_{d-1}(q) $ is 2-transitive on $ \calD(\a) $, and $G_\b^{\calD(\b)}=\PGL(\b)\cong\PGL_{2}(q) $ is 2-transitive on $ {\calD(\b)} $. It follows that $ \calD $ is $ K $-locally $ 2 $-transitive for each $ K $ with $ \PSL_d(q)\lhd K\leqslant \PGammaL_d(q) $.\vs

\end{example}

\begin{example}\label{ex:psu3q}
{\rm
Let $V=\FF_{q^2}^3$ and $(V,B)$ be a non-degenerate unitary space. The \textit{Hermitian Unitary} $U_H(q)=(\calP,\calB,\calI)$ is defined as follows\,:
\begin{itemize}
\item[] $\calP$\,: the set of isotropic $1$-subspaces of $V$,\vs
\item[] $\calB$\,: the set of non-degenerate $2$-subspaces of $V$,\vs
\item[] $\a\,\calI\,\b \Leftrightarrow \a\subset\b$, for $ \a\in\calP $, $ \b\in\calB $.
\end{itemize}

Let $ \calD=U_H(q) $. Since $ (V,B) $ is of Witt index $ 1 $, for any distinct $ \a,\a'\in\calP $, $\langle \a,\a' \rangle\in\calB $, so that, $ \calD $ is a $ 2 $-design (a linear space) with $ \lambda=1 $. By simple calculation the parameters of $ \calD $ are $(v,b,r,k,\lambda)=(q^3+1,q^2(q^2-q+1),q^2,q+1,1)$.\vs

To clarify the situation, fix a basis of $ V $ as $ u_1=e,\,u_2=d,\,u_3=f $ such that the matrix of $ B $ with respect to this basis is $$B=[B(u_i,u_j)]=\begin{bmatrix}
	0 & 0 & 1 \\ 0 & 1 & 0 \\ 1 & 0 & 0
\end{bmatrix}. $$

Let $ G=\GU_3(q) $, the unitary group consisting of matrices $M$ such that $M B \overline{M}^{\sf T}=B$, where $\overline{M}=[m_{ij}^q]$ and $M^{\sf T}$ is the transpose of $ M $. Then $ G^\calP\cong\PGU_3(q) \leqslant\Aut(\calD) $, and, by Witt's Lemma, $ G^\calP $ is $ 2 $-transitive on $ \calP $ and transitive on $ \calB $.

Let $(\a,\b)$ be a flag, where $\a=\langle e \rangle$ and $\b=\langle e ,f \rangle$.  Note that $\langle e \rangle^\bot = \langle e,d \rangle$. For any $M \in G_{\a}$, it fixes the flag $0<\langle e \rangle < \langle e,d \rangle <V$, thus of shape
$$M=\begin{bmatrix}
m_{11} & 0 & 0 \\ 
m_{21} & m_{22} & 0 \\
m_{31} & m_{32} & m_{33}
\end{bmatrix}. $$
Then $G_{\a}=Q \rtimes L$, where $Q$, $ L $, respectively, consists of matrices $ M_1 $, $M_2$ of shape
$$M_1=\begin{bmatrix}
1 & 0 & 0 \\ 
m_{21} & 1 & 0 \\ 
m_{31} & m_{32} & 1
\end{bmatrix},\,\,
M_2=\begin{bmatrix}
m_{11} & 0 & 0 \\	
0 & m_{22} & 0 \\	
0 & 0 & m_{33}\end{bmatrix}, $$
\noindent and, in particular, the derived group $Q'\triangleleft Q$ consists of matrices $ M_3 $ of shape
$$M_3=\begin{bmatrix}
1 & 0 & 0 \\ 0 & 1 & 0 \\ m_{31} & 0 & 1
\end{bmatrix}.$$

Since
\begin{align*}
M_1 B \overline{M_1}^T & = 
\begin{bmatrix}
1 & 0 & 0 \\
m_{21} & 1 & 0 \\
m_{31} & m_{32} & 1
\end{bmatrix}
\begin{bmatrix}
0 & 0 & 1 \\
0 & 1 & 0 \\
1 & 0 & 0
\end{bmatrix}
\begin{bmatrix}
1 & m_{21}^q & m_{31}^q  \\
0 & 1 & m_{32}^q \\
0 & 0 & 1
\end{bmatrix}\\
& = \begin{bmatrix}
0 & 0 & 1  \\
0 & 1 & m_{21}+m_{32}^q \\
1 & m_{21}^q+m_{32} & m_{31}^q+m_{32}m_{32}^q+m_{31}
\end{bmatrix}
=\begin{bmatrix}
0 & 0 & 1  \\ 0 & 1 & 0 \\ 1 & 0 & 0
\end{bmatrix}=B,
\end{align*}
we have $$ m_{21}+m_{32}^q=0,\mbox{\ and\ }  m_{31}^q+m_{32}m_{32}^q+m_{31}=0. $$ Note that the \textit{trace} map ${\rm Tr}: \FF_{q^2} \rightarrow \FF_{q},\, x
\mapsto x+x^q $ is surjective. Thus, $ |Q|=q^2\cdot q=q^3 $. In particular, considering $ m_{21}=m_{32}=0 $, we have $Q'\cong\E_{q}$. Moreover, each element $ M_1Q'\in Q/Q' $ is of order $ p $ and can be regarded as the entry $ m_{21} $. It follows that $Q/Q'\cong\E_{q^2}$. Meanwhile, since $M_2 B \overline{M_2}^T=B$, we have $$m_{33}=m_{11}^{-q}\in \FF^{\times}_{q^2},\mbox{\ and\ } m_{22}^{1+q}=1,$$ so that, $L\cong \GL_1(q^2) \times \GU_1(q)$. Therefore, $$G_{\a}=Q \rtimes L \cong (q.q^{2}):(\GL_1(q^2) \times \GU_1(q)).$$

On the other hand, any $ M\in G_\b $ also fixes $\b^\bot=\langle d \rangle$, so $ G_\b $ consists of matrices $ M_4 $ of shape
$$M_4=\begin{bmatrix}
m_{11}& 0 & m_{13} \\
0 & m_{22} & 0 \\
m_{31} & 0 & m_{33}
\end{bmatrix}, $$
where $\begin{bmatrix}
m_{11} &  m_{13} \\	
m_{31} &  m_{33}
\end{bmatrix}\in \GU_2(q)$, and  
since $M_4 B \overline{M_4}^T=B$, we have $m_{22}^{1+q}=1$. Hence $$G_{\b}=\GU_2(q) \times \GU_1(q).$$

\noindent Suppose further $M_4 \in G_{\a}$. Then $m_{13}=0$, so that, $$G_{\a\b}=G_\a\cap G_\b=Q' \rtimes L\cong q{:}(\GL_1(q^2) \times \GU_1(q)).$$

Thus, $|\calD(\a)|=|G_\a:G_{\a\b}|=q^2 $, $|\calD(\b)|=|G_\b:G_{\a\b}|=q+1 $, and so $ G $ is flag-transitive (locally-transitive) on $ \calD $. Further, $G_\a$ acts on ${\calD(\a)} $ with kernel $ Q'\times\Z(\GU_3(q)) $, so $G_\a^{\calD(\a)}\cong q^{2}{:}\GL_1(q^2) $ is 2-transitive on $ {\calD(\a)} $\,; $G_\b$ acts on ${\calD(\b)} $ with kernel $\Z(\GU_2(q))\times\Z(\GU_1(q)) $, so $G_\b^{\calD(\b)}\cong\PGU_2(q)\cong\PGL_2(q)$ is 2-transitive on $ {\calD(\b)} $. Hence the group $ G^\calP\cong\PGU_3(q) $ is locally 2-transitive on $ \calD $.\vs
}
\end{example}

\begin{example} \label{2^d_design}
{\rm According to \cite{2^d_design}, for each $ m\geq 2 $, there exist two complementary symmetric $ 2 $\,-\,$ (2^{2m},2^{2m-1} \pm 2^{m-1},2^{2m-2}\pm 2^{m-1}) $ designs $\calD=\scrS^{\pm}(2m)$, where $\Aut(\calD)=\ZZ_2^{2m}{:}\Sp_{2m}(2)$ is 2-transitive on both $ \calP$ and $\calB $, and further, letting $ G=\Aut(\calD) $ and $(\a,\b)\in\calP\times\calB$ be a flag, then\,:
\begin{align*}
&G = \mathbb{Z}_2^{2m}{:}\Sp_{2m}(2), \\
&G_{\a} \cong G_\b \cong \Sp_{2m}(2), \\	&G_{\alpha\beta} \cong \GO^\pm_{2m}(2).
\end{align*}	
Thus, in both cases, $|\calD(\a)|=|G_\a:G_{\a\b}|$, $|\calD(\b)|=|G_\b:G_{\a\b}| $, and so $ G $ is flag-transitive (locally transitive) on $ \calD $. Further, $ G_\a^{\calD(\a)}\cong G_\a\cong\Sp_{2m}(2) $ is 2-transitive on $ \calD(\a) $ (since it is with stabilizer $ \GO^\pm_{2m}(2) $). Similarly, $ G_\b^{\calD(\b)}\cong G_\b $ is 2-transitive on $ \calD(\b)$. Hence $ G $ is locally 2-transitive on $ \calD $.\vs
}
\end{example}

\begin{example}\label{AG}
{\rm Let $V=\mathbb{F}_q^d$, $ d\geq 2 $. The {\it affine space} $\AG_i(d,q)$ is defined as follows\,:
\begin{itemize}
\item[] $\calP=\{ v\, |\, v \in V \}$,\vs
\item[] $\calB_i= \{U+v \ |\ U<V \text{ is an $i$-subspace\ and\ } v \in V\}$,\vs
\item[] $\AG_i(d,q)=(\calP,\calB_i,\calI)$,\, $\a\,\calI\,\b \Leftrightarrow \a\subset\b$, for $ \a\in\calP $, $ \b\in\calB_i $.
\end{itemize}

Let $ \calD=\AG_i(d,q) $. Note that each block is a coset of some $i$-subspace of $V$. By \cite[Exam.\,4.2]{locallyprimitive}, $\calD$ is a locally primitive design with $\Aut(\calD)=\A\GammaL_d(q)$. Let $ G=\ASL_d(q)\leqslant\Aut(\calD) $. We consider the special case $ i=1 $ or $ d-1 $.
		
{\bf (1)} For $\calD=\AG_1(d,q)$, $(v,b,r,k,\lambda)=(q^d,\frac{q^{d-1}(q^d-1)}{q-1},\frac{q^d-1}{q-1},q,1)$. Let $ (\a,\b) $ be a flag, where $ \a={\0}\in V $ and $ \b=U< V $. Then we have
\begin{align*}
&G = \ASL_d(q), \\
&G_{\alpha} \cong \SL_d(q), \\
&G_{\beta} \cong q{:}(q^{d-1}{:}\GL_{d-1}(q)),\\
&G_{\alpha\beta}\cong q^{d-1}{:}\GL_{d-1}(q),
\end{align*}
\noindent where $ G_\b $ contains the translations with respect to $ U $, as well as the stabilizer $ \SL_d(q)_U $. Note that $ G $ is  locally transitive on $ \calD $, since $ |{\calD(\a)}|=|G_\a:G_{\a\b}|=\frac{q^d-1}{q-1} $, and $|\calD(\b)|=|G_\b:G_{\a\b}|=q $. Further, $G_\a^{\calD(\a)}\cong\PSL_{d}(q) $ is 2-transitive on $ \calD(\a) $\,; $G_\b^{\calD(\b)}=\AGL(U)\cong\AGL_{1}(q) $ is 2-transitive on $ {\calD(\b)} $. It follows that $ \calD $ is $ K $-locally $ 2 $-transitive for each $ K $ with $ \ASL_d(q)\lhd K\leqslant \AGammaL_d(q) $. We remark that, if $ q=2 $, then $ \binom{v}{k}=\binom{2^d}{2}=2^{d-1}(2^d-1)=b $, and so $ \calD $ is trivial.\vs

{\bf (2)} For $\calD=\AG_{d-1}(d,q)$,
$(v,b,r,k,\lambda)=(q^d,\frac{q(q^d-1)}{q-1},\frac{q^d-1}{q-1},q^{d-1},\frac{q^{d-1}-1}{q-1})$. Let $ (\a,\b) $ be a flag, where $ \a={\0}\in V $ and $ \b=U<V $. Then we have
\begin{align*}
&G = \ASL_d(q), \\
&G_{\alpha} \cong \SL_d(q), \\
&G_{\beta} \cong q^{d-1}{:}(q^{d-1}{:}\GL_{d-1}(q)),\\
&G_{\alpha\beta} \cong q^{d-1}{:}\GL_{d-1}(q),
\end{align*}
\noindent and, similarly, $G_\a^{\calD(\a)}\cong\PSL_{d}(q) $ is 2-transitive on $ \calD(\a) $, and $G_\b^{\calD(\b)}=\AGL(U)\cong\AGL_{d-1}(q) $ is 2-transitive on $ {\calD(\b)} $. Hence $ \calD $ is $ K $-locally $ 2 $-transitive for each $ K $ with $ \ASL_d(q)\lhd K\leqslant \AGammaL_d(q) $.
}
\end{example}

\section{The symmetric case} \label{sec-Kantor}
In this section we study case {\bf(A)}. The symmetric 2-transitive designs were determined by Kantor \cite{Kantor2tr}. We find that there is only one single design among them that is not locally $ 2 $-homogeneous.

\begin{proposition}\label{sym}
Let $\calD=(\calP,\calB,\calI)$ be a symmetric $2$-$(v,k,\lambda)$ design with $ G\leqslant\Aut(\calD) $ being $ 2 $-transitive on $ \calP $. Then either {\rm(1)} or {\rm(2)} occurs.\vs
\noindent {\rm(1)} $G$ is locally $2$-transitive on $\calD$. One of the following holds{\rm\,:}
\begin{itemize}
	\item[\rm(1.1)] $\calD=\PG(d-1,q)$ or $\overline{\PG(d-1,q)}$, and $\PSL_d(q)\lhd G \leqslant \PGammaL_d(q)${\rm\,;}
	\vskip1pt
	\item[\rm(1.2)] $\calD=\scrS^{\pm}(2m)$, and $G=2^{2m}{:}\Sp_{2m}(2)${\rm\,;}
	\vskip1pt
	\item[\rm(1.3)] $ \calD $ and $ (G,G_\a,G_{\a\b},G_\b) $ are listed in {\rm Table \ref{table_sym}}, with $(\a,\b)\in\calP\times\calB$ a flag.
\end{itemize}
\noindent {\rm(2)}  $G$ is not locally $2$-homogeneous on $\calD$. One of the following holds{\rm\,:}
\begin{itemize}
	\item[\rm(2.1)] $\calD=\scrS^{\pm}(2m)$, and $G=q^6{:}\G_2(q).\calO < 2^{2m}{:}\Sp_{2m}(2)$, where $q^6=2^{2m}$, $\calO\leqslant \langle \phi \rangle \cong \ZZ_{\frac{m}{3}}$, and $\phi$ is the field automorphism{\rm\,;}
	\vskip1pt
	\item[\rm(2.2)] $\calD$ is the unique $ 2 $-$ (176,50,14) $ design.
\end{itemize}
\end{proposition}

\begin{proof}
First, $v\ne2k$, as if not, $\lambda(2k-1)=\lambda(v-1)=k(k-1)$ and $k(k-1)|\lambda$, not possible. Thus, either $ v>2k $\,; or $v<2k$, $v>2(v-k)$. According to Kantor \cite{Kantor2tr}, either $ \calD $ or the complement $ \overline\calD $ is one of\,:\vskip1pt
- the projective space $ \PG(d-1,q) $,\vskip1pt 
- the unique Hadamard design with $(v,k,\lambda)=(11,5,2) $,\vskip1pt 
- the unique design with $(v,k,\lambda)=(176,50,14) $,\vskip1pt 
- the design $ \scrS^{-}(2m) $ given in \cite{2^d_design} on $ 2^{2m} $ points, \vskip1pt
\noindent where, respectively, $ \Aut(\calD) $ is $ \PGammaL(d,q) $, $ \PSL(2,11) $, $ \HS $ and $ 2^{2m}{:}\Sp_{2m}(2) $. Moreover, each of these designs is \textit{flag-transitive} and \textit{self-dual} (their incidence graphs are known to be \textit{distance-transitive}). 
Further, if $G$ is almost simple, then the first three cases occur;
if $G$ is an affine group, then the last case occurs.

By Theorem 4.3 of \cite{design}, for a $G$-point transitive symmetric design, the rank of $G$ on points equals the rank of $G$ on blocks.
We remark that Kantor's result relies on 
an important observation of this, that is, $ G $ is not only $ 2 $-transitive on $ \calP $ but also $ 2 $-transitive on $ \calB $, and, further, these two $ 2 $-transitive permutation representations are inequivalent.\vs 

{\bf(a)} Let $\calD=\PG(d-1,q)$ or $\overline{\PG(d-1,q)}$, where $ \Aut(\calD)=\PGammaL(d,q) $. By Example \ref{ex:Pro-G}, $ \calD $ is $G$-locally $ 2 $-transitive for $\PSL_d(q)\lhd G \leqslant \PGammaL_d(q)$. This gives part (1.1). 

We consider other possibilities for $ G $. Note that $ G $ is an almost simple group that has two inequivalent $2$-transitive representations of degree $ v $ with $\soc(G)\ne\PSL_d(q)$. Let $ \soc(G)=T $. By checking the finite 2-transitive groups in \cite[Thm.\,5.3]{as2trans}, we have that $ T $ also has two inequivalent $2$-transitive representations of degree $ v $, and consequently, $ T $ is one of\,: \vs
$\A_6\, (v=6),\ \PSL_2(11)\, (v=11),\ \M_{12}\, (v=12),\ \A_7\, (v=15),\ \HS\,(v=176)$.\vs
\noindent It is easy to check that $\frac{q^d-1}{q-1}\ne 6, 11, 12, 176$ for prime power $ q $ and $ d>2 $. For the particular case $\frac{2^4-1}{2-1}=15$, we have $ G=\A_7<\PGL_4(2) $. By $ \MA $, there exist subgroups $ H,K\leqslant G $ such that \[\mbox{$ H\cong K\cong \PSL_2(7) $ and $ H\cap K\cong \S_4 $ or $ 7{:}3 $.}\] The triple $ (G,H,K) $ satisfies the conditions of Construction \ref{cons}, thereby giving rise to a $ G $-locally $ 2 $-transitive symmetric design with $ (v,k,\lambda)=(15,7,3) $ or $ (15,8,4) $. By Kantor's list, it is either $\PG(3,2)$ or $\overline{\PG(3,2)}$. This gives Row 2 of Table \ref{table_sym}.\vs

{\bf(b)} Let $ \calD=\scrS^{-1}(2m) $, where $\Aut(\calD)=2^{2m}{:}\Sp_{2m}(2)$. By Example \ref{2^d_design}, $\calD$ is $ G $-locally $ 2 $-transitive for $G=\Aut(\calD)$. This gives part (1.2).

We consider other possibilities for $ G $. 
Note that $ G $ has two inequivalent $2$-transitive representations of degree $ v=2^{2m} $, and $G$ is affine such that $ G=2^{2m}{:}G_\a $, where $G_{\a}<\Sp_{2m}(2)$. By \cite[Section 2]{G2}, $\G_2(q)$ is a subgroup of $\Sp_{6}(q)$, where $q=2^n$.
By \cite[Appendix~1]{ha2trans} and Kantor's discussion, either $v=2^4$, and $G_\a=\A_6=\Sp_4(2)'$ \,; or
$v=2^{2m}=q^6$, and $\G_2(q)' \triangleleft G_\a$.\vs

Suppose that $ v=2^4 $, and $G_\a=\A_6=\Sp_4(2)'$. Recall that $ \calD $ is flag-transitive (locally transitive). The only transitive permutation representation of $ \Sp_4(2) $ of degree $ 6 $ or $ 10 $ is the $ 2 $-transitive representation with stabilizer $ \GO_4^\pm(2) $. Then, by \cite{ATLAS}, either\vskip1pt $(v,k,\lambda)=(16,6,2)$,\, $(G,G_{\a},G_{\a\b})=(2^4{:}\A_6,\A_6,\A_5)$\,; or\vskip1pt
$(v,k,\lambda)=(16,10,6)$,\, $(G,G_{\a},G_{\a\b})=(2^4{:}\A_6,\A_6,3^2{:}4)$.\vskip1pt
\noindent In both cases, $ G $ is locally 2-transitive on $ \calD $. This gives Row 3 of Table \ref{table_sym}.

Suppose that $ v=2^{2m}=q^6 $, and $\G_2(q)' \triangleleft G_\a<\Sp_{2m}(2)$. Only if $ q=2 $, $\G_2(2)'=\PSU_3(3)$ has a $ 2 $-transitive representation. In this case, $ G=2^6{:}\PSU_3(3).\calO $, where $ \calO =1, 2$. By $ \MA $, there exist subgroups $ H,K\leqslant G $ such that \[\mbox{$ H\cong K\cong \PSU_3(3).\calO $ and $ H\cap K\cong 3_+^{1+2}{:}8.\calO $.}\] Note that the triple $ (G,H,K) $ satisfies the conditions of Construction \ref{cons}, thereby giving rise to a $ G $-locally $ 2 $-transitive symmetric $ 2 $-$ (64,28,12) $ design, that is, $\scrS^{-}(6)$. This gives Row 4 of Table \ref{table_sym}.
For the case $q=2^{\frac{m}{3}}>2$, $\G_2(q)=\G_2(q)'$ is simple and has no $2$-homogeneous representation. By the order of $\G_2(q)$, the value of $k$, and the subgroups of $\G_2(q)$ on \cite{G2} and the Table 8.30 of \cite{low-dimensional}, we have
either\vskip1pt $\calD=\scrS^{-}(2m)$, and $(G,G_{\a},G_{\a\b})=(q^6{:}\G_2(q).\calO,\G_2(q).\calO,\SU_3(q){:}2.\calO)$\,; or\vskip1pt
$\calD=\scrS^{+}(2m)$, 
and $(G,G_{\a},G_{\a\b})=(q^6{:}\G_2(q).\calO,\G_2(q).\calO,\SL_3(q){:}2.\calO)$, \vskip1pt
\noindent where $\calO\leqslant \langle \phi \rangle \cong \ZZ_{\frac{m}{3}}$, and $\phi$ is the field automorphism. We have part (2.1).

\vs
	 
{\bf(c)} At last, we treat the four sporadic designs. Recall that they are flag-transitive (locally transitive). Let $ G=\Aut(\calD) $ and $ (\a,\b)\in\calP\times\calB $ be a flag. 

For $ \calD $ with $(v,k,\lambda)=(11,5,2)$ or $(11,6,3)$, we have $ G=\PSL_2(11) $, and $G_\a\cong G_\b \cong \A_5 $. The only transitive permutation representation of $ \A_5 $ of degree $ 5 $ or $ 6 $ is, respectively, the $ 2 $-transitive representation with stabilizer $ \A_4 $ or $\D_{10}$. Thus, $G$ is locally $ 2 $-transitive on $ \calD $. This gives Row 1 of Table \ref{table_sym}. Note that no group smaller than $ \PSL_2(11) $ acts 2-transitively on $ 11 $ points. 

For $ \calD $ with $(v,k,\lambda)=(176,50,14)$ or $(176,126,90)$. We have $ G=\HS $, and $G_\a\cong G_\b \cong \PSU_3(5){:}2 $. The only transitive permutation representation of $ \PSU_3(5){:}2 $ of degree $ 126 $ is the $ 2 $-transitive representation with stabilizer $ 5_+^{1+2}{:}8{:}2 $. Thus the latter design is $ G $-locally 2-transitive. This gives Row 5 of Table \ref{table_sym}. Note that no group smaller than $ \HS $ acts 2-transitively on $ 176 $ points. On the other hand, $\PSU_3(5){:}2$ has no $2$-homogeneous representation of degree $ 50 $. We have part (2.2). 
\end{proof}

\section{The affine case}\label{sec-affine}
Let $\calD=(\calP,\calB,\calI)$ be a $G$-locally $2$-homogeneous design. In this section we study case {\bf(B)}, where $ G $ is an affine 2-homogeneous group on $ \calP $ (thus, $ v=p^f $ is a prime power). By Wagner-Kantor's Lemma, there are two subcases to treat\,: 

- $G \leqslant \AGammaL_1(p^f)$\,; 

- $G \not\leqslant \AGammaL_1(p^f)$, so that, $ G $ is $ 2 $-transitive on $ \calP $, and by Proposition \ref{sym} we only need to consider non-symmetric designs.\vs\vs

\subsection{$G \leqslant \AGammaL_1(p^f)$}{~}\vskip2pt
There are exactly three $G$-locally 2-homogeneous designs with $G \leqslant \AGammaL_1(p^f)$. These designs are isomorphic to certain projective or affine spaces. 

\begin{proposition}\label{lem:1-dim}
Let $\calD=(\calP,\calB,\calI)$ be a $G$-locally $2$-homogeneous design with $v=p^f$ a prime power and $G\leqslant \AGammaL_1(p^f)$. Let $ (\0,\b)\in\calP\times\calB $ be flag. Then one of the following holds{\rm\,:}
\begin{itemize}
\item [\rm(1)] $\calD=\PG(2,2)$, and $(G,G_{\0},G_{\0 \b},G_{\b})=(\ZZ_{7}{:}\ZZ_3,\ZZ_{3},1,\ZZ_{3})${\rm\,;}\vskip0.03in
\item [\rm(2)] $\calD=\AG_2(3,2)$, and $(G,G_{\0},G_{\0 \b},G_{\b})=(\AGammaL_1(8),\GammaL_1(8),\ZZ_{3},\ZZ_2^2{:}\ZZ_{3})${\rm\,;}\vskip0.03in
\item [\rm(3)] $\calD=\AG_1(2,4)$, and $(G,G_{\0},G_{\0 \b},G_{\b})=(\AGammaL_1(16),\GammaL_1(16),\ZZ_{3}{:}\ZZ_4,\ZZ_2^2{:}(\ZZ_{3}{:}\ZZ_4))$.
\end{itemize}
\noindent Further, $ G $ is locally $ 2 $-transitive on $ \calD $ only in {\rm(3)}.
\end{proposition}

\begin{proof}
The proof proceeds by filtering for suitable parameters. The following restriction on parameters follows by basic property and Lemma \ref{k<=v/2}\,:  $$ \lambda(v-1)=r(k-1),\ 3\leq k\leqslant r,\ k\notin\{v-2,v-1\}.\ \ \ \ (\ast) $$
\noindent Meanwhile, the stabilizer  $G_\0\leqslant \GammaL_1(p^f)$ is a metacyclic group, and thus the $2$-homogeneous group $G_\0^{\calD(\0)}$ of degree $ r $ is either	$\ZZ_r{:}\ZZ_{r-1}$ or $ \ZZ_r{:}\ZZ_{\frac{r-1}{2}}$, with $r \mid (p^f - 1)$, and, respectively, $r-1 \mid f$ or $\frac{r-1}{2} \mid f$. In both cases, $r\leq 2f+1 $. By $ (\ast) $, we have\,:
$$ p^f-1=v-1 \leq r(k-1) \leq (2f+1)\cdot 2f.\ \ \ \ \ (\ast\ast)$$

{\bf(a)} Suppose $ r=3 $. By Lemma \ref{k<=v/2}, $ v=7 $, $\calD=\PG(2,2)$ and $ \Aut(\calD)=\PGL(3,2) $. Note that $ G $ is 2-homogeneous on $ \calP $ with $ G\leqslant\AGL_1(7)\cap\PGL(3,2)$. Thus, $ G=\ZZ_7{:}\ZZ_3 $, and $ G_\0\cong G_\b\cong\ZZ_3 $, $ G_{\0\b}=1 $. Note that $ G $ is locally 2-homogeneous but not locally $ 2 $-transitive on $ \calD $. Part {\rm(1)} is satisfied.\vs
	
{\bf(b)} Suppose $r \geq 4$. Then $ f\geq 2 $. The only solutions for $(\ast\ast)$ are \[(p,f)=(3,2),\,(3,3) \mbox{\ or\ } (2,f) \mbox{\ with\ }2 \leq f \leq 8.\] Note that, $ (p,f)=(3,2) $ yields $G_\0\leqslant \GammaL(1,9)=\ZZ_8{:}\ZZ_2$ and $r \div 8$, so that, $r=4,8$, $G_\0^{\calD(\0)}\cong\ZZ_r{:}\ZZ_{r-1}$\,; $ (p,f)=(3,3) $ yields $r \div 26$, so $r \geqslant 13> 2f+1$\,; $ (p,f)=(2,2) $ yields $r\div 3$. All these cases are impossible. For the remaining cases, it follows from $r^2\geq rk>\lambda v\geqslant v=2^f$ that $r$ is a prime power satisfying \[\mbox{$2^{\frac{f}{2}} <r \leqslant 2f+1$ and $r \div 2^f-1$.}\] We conclude that $(f,r)\in\{(3,7),(4,5),(6,9),(8,17)\}$. Further, if $r\not\equiv3\,(\mod 4)$, then $G_\0^{\calD(\0)}$ is 2-transitive, so $r-1 \div f $. Thus, $(f,r)\ne(6,9)$ or $(8,17)$.\vs
	
{\bf(b.1)} Assume $(p,f,r)=(2,3,7)$. Then $ v=8 $, $ G\leq \AGammaL_1(8)=\ZZ_2^3{:}(\ZZ_7{:}\ZZ_3) $. By $ (\ast) $, $ k\in\{3,4,5\} $\,; further, $ vr=bk $ yields $ k=4$, $b=14 $. Thus, $ (v,b,r,k,\lambda)=(8,14,7,4,3) $. Since $G_\0$ is $2$-homogeneous on $r=7$ points, we have $ 7\cdot 3\div |G_\0| $, so $G=\AGammaL_1(8)$. Note that there exist $ H,K\leqslant G $ such that \[\mbox{$ H=\GammaL_1(8)=\ZZ_7{:}\ZZ_3,\, K=\ZZ_2^2{:}\ZZ_3 $,\, and $ H\cap K=\ZZ_3 $.}\] By Construction \ref{cons}, the triple $ (G,H,K) $ induces a design $ \calD $ with above parameters, which, by $ \MA $, is isomorphic to $ \AG_2(3,2) $. Note that $ G $ is locally 2-homogeneous but not locally $ 2 $-transitive on $ \calD $. Part {\rm(2)} is satisfied.\vs\vs
	
{\bf(b.2)} Assume $(p,f,r)=(2,4,5)$. Then $v=16$, $G\leq\AGammaL_1(16)=\ZZ_2^4{:}(\ZZ_{15}{:}\ZZ_4)$. The only pair satisfying $ (\ast) $ is $ (k,\lambda)=(4,1) $. Thus, $(v,b,r,k,\lambda)=(16,20,5,4,1)$. Since $G$ is $2$-transitive on $v=16$ points, and $G_\0$ is $2$-transitive on $r=5$ points, we have $15\cdot 4\div |G_\0|$, so $ G=\AGammaL_1(16) $. Note that there exist $ H,K\leqslant G $ such that \[\mbox{$ H=\GammaL_1(16)=\ZZ_{15}{:}\ZZ_4,\, K=\ZZ_2^2{:}(\ZZ_3{:}\ZZ_4) $,\, and $ H\cap K=\ZZ_3{:}\ZZ_4 $.}\] By Construction \ref{cons}, the triple $ (G,H,K) $ induces a design $ \calD $ with above parameters, which, by $ \MA $, is isomorphic to $\AG_1(2,4)$. Note that $ G $ is locally 2-transitive on $ \calD $. Part {\rm(3)} is satisfied.
\end{proof}

A construction of $\AG_1(2,4)$ is as below. 

\begin{example}\label{exam:AGammaL}
{\rm Let $G=\AGammaL_1(16)$ act on $\FF_{16}$ in its natural action. Write the field $\FF_{16} = \{0\} \cup \FF_{16}^\times=\{0\}\cup \langle \omega \rangle$, and let $ \b = \{0, 1, \omega^5, \omega^{10}\}$ be with $\b\cong \FF_4 $. Define $\calD = (\calP,\calB,\calI)$ with  $\calP = \FF_{16}$, $\calB = \{\, \b^g \mid g \in G \,\}$, and $\a\,\calI\,\b \Leftrightarrow \a\subset\b$. Then $ \calD $ is isomorphic to $ \AG_1(2,4) $. Further, $G$ is locally $2$-transitive on $ \calD $, where $ (G,\,G_\0,\,G_{\0\b},\,G_\b)=(\AGammaL_1(16),\,\GammaL_1(16),\,\ZZ_{3}{:}\ZZ_4,\,\ZZ_2^2{:}(\ZZ_3{:}\ZZ_4)) $.\vs\vs\vs
}
\end{example}

\subsection{$G \not\leqslant \AGammaL_1(p^f)$}{~}\vskip2pt

In this part, we study designs that satisfy the following hypothesis.

\begin{hypothesis}\label{hypo-1}
{\rm Let $\calD=(\calP,\calB,\calI)$ be a $G$-locally 2-homogeneous non-symmetric design such that $G$ is an affine 2-transitive group on $ \calP $, and further, $G\not\leqslant\AGammaL_1(v)$, where $v=p^f $, $ p $ is a prime. Let $(\0,\b) \in \calP \times \calB$ be a flag, and write $ G=V{:}G_\0=\ZZ_p^f{:}G_\0$, where $G_\0 \leqslant \GL_f(p)$. Note that $ r\geq 4 $ by Lemma \ref{k<=v/2}.} 
\end{hypothesis}

The main result is as follows.

\begin{proposition}\label{Affine}
Under {\rm Hypothesis \ref{hypo-1}}, one of the following holds{\,:}
\begin{itemize}
\item[\rm(1)] $\calD=\AG_1(a,q)$ or $\AG_{a-1}(a,q)$, and $\ASL_a(q)\lhd G\leqslant\AGammaL_a(q)$, where $a\geq2$, $v=q^a>4${\rm\,;}\vs
\item[\rm(2)] $\calD =\AG_3(4,2)$, and $(G,G_{\0},G_{\0\b},G_{\b})=(2^4{:}\A_7,\A_7,\PSL_3(2),2^3{:}\PSL_3(2))$.
\end{itemize}
Further, $ G $ is locally $ 2 $-transitive on $ \calD $ in both cases.
\end{proposition}

The proof of the this result proceeds by analyzing affine 2-transitive groups, which were classified in \cite{HERING1985151}; see also \cite[Appendix~1]{ha2trans}. There are three {\it infinite classes} (not contained in $ \AGammaL_1(p^f) $), as well as some {\it extraspecial classes} and {\it exceptional classes}. These classes are respectively treated in Lemmas \ref{lem:SL}\,-\,\ref{G_0_exceptional}. Before that, we establish the next lemma. It is indeed part of \cite[Thm.\,1.3]{locallyprimitive}, which indicates that a design satisfying Hypothesis \ref{hypo-1} can only be a subdesign of some affine space $ \AG_i(f,p) $.

\begin{lemma}\label{lem:pty-1}
Under {\rm Hypothesis \ref{hypo-1}}, there exists a nontrivial $G_{\0\b}$-invariant $ i $-subspace $M<V=\ZZ_p^f$ such that $G_\b=M{:}G_{\0\b}$. In particular, $ k=p^i $, and $ \calD $ is a subdesign of $ \AG_i(f,p) $.
\end{lemma}

\begin{proof}
Since $ G $ is locally primitive on $ \calD $, it follows that $G_{\0\b}$ is a maximal subgroup of both $G_{\0}$ and $G_{\b}$, so that, $G_{\b}=\langle G_{\0\b}, g \rangle$ for some $g \in G_{\b} \setminus G_{\0\b}$. Write $ G=V{:}G_\0 $ and $g={\bf v}h$ with ${\bf v} \in V$, $ h\in G_{\0} $. Thus, either $h \in G_{\0} \setminus G_{\0\b}$\,; or $h \in G_{\0\b}$, $ {\bf v}\ne 0 $. For the former case, we have $$VG_{\b}=V\langle G_{\0\b}, g \rangle=V\langle G_{\0\b}, h \rangle=VG_{\0}=G,$$ so that, $b = |G : G_\b| \leq |V| = v$. Note that $v \leq b$. Then $v=b$ and $\calD$ is symmetric, not in our case. Hence the latter case holds, i.e., $h \in G_{\0\b}$, $ {\bf v}\ne 0 $, and
$$ G_{\b} = \langle G_{\0\b}, g \rangle = \langle G_{\0\b}, {\bf v} \rangle = \langle {\bf v}^{G_{\0\b}} \rangle {:} G_{\0\b}. $$
Let $M= \l {\bf v}^{G_{\0\b}}\r$. Then $ M $ is a non-trivial subspace of $ V $ as $ {\bf v}\ne\0 $. If $M=V$, then $G_{\b}$ contains all translations, and $\b = \calP$ since $\0$ is incident with $\b$, contrary to the assumption that $ \calD $ is incomplete $ (k<v) $. Hence $1<M<V$, and $ G_\b=M{:}G_{\0\b} $, $ k=|G_\b:G_{\0\b}|=p^i $, where $ i=\dim M $. Further, $ M\lhd G_\b $ is transitive on $ \calD(\b)=\b $. Thus, $ \b =\0^M=M $. Now, $ \calB $ is the orbit $ M^G $ which consists of the cosets of subspaces lying in $ M^{G_\0} $. Hence $ \calD $ is a subdesign of $ \AG_i(f,p) $.
\end{proof}

\begin{lemma}\label{lem:SL}
Under {\rm Hypothesis \ref{hypo-1}}, if $ G $ is of infinite classes, then $\calD$ is $\AG_1(a,q)$ or $\AG_{a-1}(a,q)$ and $ \SL_a(q)\lhd G_\0\leq\GammaL_a(q)$, where $a\geq2$, $v=q^a>4$. Further, $ G $ is locally $ 2 $-transitive on $ \calD $.
\end{lemma}

\begin{proof}
By \cite[Appendix~1]{ha2trans}, we have three infinite classes to consider\,:
	
{\bf(a)} $ \SL_a(q)\lhd G_\0\leq\GammaL_a(q) $, where $ p^f=q^a $ and $ a\geq 2 $\,;
	
{\bf(b)} $ \Sp_{2m}(q)\lhd G_\0 $, where $ p^f=q^{2m} $, $ m\geq 2 $ and $(m,q)\ne(2,2)$\,;
	
{\bf(c)} $ \G_2(q)'=\G_2(q)\lhd G_\0 $, where $ p^f=q^6 $, $ p=2 $ and $ q>2 $.\vs
	
Note that, the group $ \SL_a(q) $, $ \Sp_{2m}(q) $ or $ \G_2(q)' $, respectively in class {\bf(a)}\,-\,{\bf(c)}, is transitive on $V\setminus\{\0\}$, so it acts non-trivially on $\calD(\0)$. In particular, for {\bf(b)}, {\bf(c)}, $G_\0^{\calD(\0)}$ is an almost simple 2-transitive group with socle $ \PSp_{2m}(q) $ or $ \G_2(q) $. This does not hold, because $\G_2(q)$ admits no $2$-transitive permutation representation for $ q>2 $, and meanwhile, $ \PSp_{2m}(q) $ has a $2$-transitive permutation representation only if $ q=2 $, in which case the stabilizer $\GO_{2m}^{\pm}(2)$ acts irreducibly on $V = \ZZ_2^{2m}$, so $ G_{\0\b} $ also acts irreducibly on $V$, contradicting Lemma~\ref{lem:pty-1}.\vs
	
We then consider the linear class {\bf(a)}. If the center $\Z(\SL_a(q))\lhd G_\0 $ is transitive on $\calD(\0)$, then $r \leqslant q-1$ and $q^2-1 \leqslant v-1\leqslant\lambda(v-1)=r(k-1)<r^2 \leqslant (q-1)^2$, not possible, so that, it acts trivially on $\calD(\0)$. Suppose $ (a,q)\ne (2,2),\,(2,3) $. Then $G_\0^{\calD(\0)}$ is an almost simple 2-transitive group with socle $ \PSL_a(q) $ of degree $ r $.\vs
	
For the special cases, if $ (a,q)=(2,2) $, then $ k<v=4 $, leading to trivial case by Lemma \ref{k<=v/2}. If $(a,q)=(2,3)$, $ G_\0\leq\GL_2(3) $, so $G_\0^{\calD(\0)}$ is affine with degree $ r=4 $, since $ r\geq 4 $ and $ r\div |\PSL_2(3)|=12 $. This can be subsumed under (a.1) for handling.\vs
	
{\bf(a.1)} Assume $r={q^a-1\over q-1}$. The group $ G_\0^{\calD(\0)} $ with socle $ \PSL_a(q) $ has only two $ 2 $-transitive representations of degree $ {q^a-1\over q-1} $. The set $ \calD(\0) $ is regarded as the set of either 1-subspaces or hyperplanes of $ V=\FF_q^a $. Thus, $ \calB$ is equal to the orbit $\calD(\0)^V $, i.e., the set of all cosets of subspaces in $ \calD(\0) $. That is, $\calD$ is either $\AG_1(a,q)$ or $\AG_{a-1}(a,q)$.
By Example \ref{AG}, $\calD $ is $ G $-locally $ 2 $-transitive for each $ G $ with $ \ASL_d(q)\lhd G\leqslant \AGammaL_d(q) $. \vs
	
{\bf(a.2)} Assume $r\ne {q^a-1\over q-1}$. By \cite[Thm.\,5.3]{as2trans}, the cases are as follows.
\[\mbox{\small $\begin{array}{c|lllllllll}\hline
G_\0^{(\infty)}&\SL_2(4)&\SL_2(5)&\SL_2(7)&\SL_2(9)&\SL_3(2)&\SL_4(2)&\SL_2(11)&\SL_2(8)\\ \hline
r&6&5&7&6&8&8&11&28\\ \hline
v=q^a&4^2&5^2&7^2&9^2&2^3&2^4&11^2&8^2 \\ \hline
\end{array}$}\]
Further, since $r^2>rk>\lambda v \geqslant v=q^a$, only the following four need to consider\,:
\[(r,v)=(6,4^2),\,(8,2^3),\,(8,2^4),\,(28,8^2).\]
In the first three cases, there is no suitable value for $ k $ satisfying $ \lambda(v-1)=r(k-1) $ and $ k<r $. In the last case, $ 63\lambda=28(k-1) $, so $ 9\lambda=4(k-1) $, and so $ k=10 $ or $ 19 $, which contradicts that $k$ is a $ 2 $-power by Lemma~\ref{lem:pty-1}.
\end{proof}

\begin{lemma} \label{G_0_extraspecial}
Under {\rm Hypothesis \ref{hypo-1}}, the group $G$ is not of extraspecial classes.
\end{lemma}

\begin{proof}
By \cite[Appendix~1]{ha2trans}, there are a few exceptional small 2-transitive groups $G=V{:}G_\0=\ZZ_p^f{:}G_\0$ such that $G_\0\leqslant \N_{\GL_d(p)}(R)$, where $R$ is irreducible on $V$ with one of the following holding\,:
\begin{itemize}
\item[(1)] $R=\Q_8\triangleleft G_\0$, and $|V|=v=5^2$, $7^2$, $11^2$ or $23^2$\,; or\vskip0.03in
\item[(2)] $R=2^{1+4}=\Q_8\circ\D_8\triangleleft G_\0 $, and $|V|=v=3^4$.
\end{itemize}

Note that $R \triangleleft G_\0$ and $G_\0$ is primitive on $\calD(\0)$.
If $R$ acts trivially on $\calD(\0)$, then $R \leqslant G_{\0\b}$, and thus $G_{\0\b}$ is irreducible on $V$, too. It is impossible by Lemma~\ref{lem:pty-1}. Hence $R$ acts transitively on $\calD(\0)$. For (1), $\Z(R)=\ZZ_2$ acts trivially on $\calD(\0)$, so $r \div |R/\Z(R)|=4$, and thus $r=4$. Then $ 4^2=r^2\geq rk>\lambda v $, a contradiction. For (2), $\Z(R)=\ZZ_2 $ acts trivially on $\calD(\0)$, so $r \div |R/\Z(R)|=16$.
Since $r^2>v=3^4$ and $\lambda(v-1)=r(k-1)$, it follows that $r=16$, $5\lambda=k-1$, and so $k=6$ or $11$, which contradicts that $k$ is a $ 3 $-power by Lemma~\ref{lem:pty-1}.
\end{proof}

\begin{lemma} \label{G_0_exceptional}
Under {\rm Hypothesis \ref{hypo-1}}, if $G$ is of exceptional classes, then $\calD$ is the unique design $\AG_3(4,2)$ and $(G,G_{\0},G_{\0\b},G_{\b})=(2^4{:}\A_7,\A_7,\PSL_3(2),2^3{:}\PSL_3(2))$. Further, G is locally $ 2 $-transitive on $ \calD $.
\end{lemma}

\begin{proof}
By \cite[Appendix~1]{ha2trans}, we have the following exceptional classes to consider.\vs

{\bf(a)} $ |V|=2^6 $ and $ G_\0 =\G_2(2)'=\PSU_3(3)<\Sp_6(2) $. In this case, the $ 2 $-transitive permutation representation of $ G_\0 $ is of degree $ r=28 $. By Lemma \ref{lem:pty-1}, $ k $ is a $ 2 $-power, and it is with $3 \leqslant k <r$, so $k=4$, $8$ or $16$. It follows by $ r(k-1)=\lambda(v-1)=63\lambda $ that $ \lambda $ is a non-integer. Hence this case cannot occur.\vs

{\bf(b)} $ |V|=p^d\in\{9^2,11^2,19^2,29^2,59^2\} $ and $ \SL_2(5)=G_\0^{(\infty)}\lhd G_\0<\SL_2(p^{d/2}) $. In this case, $r=5$ or $6$. Then $v \leqslant \lambda v < rk < r^2 \leqslant 6^2$, not possible.\vs

{\bf(c)} $ |V|=2^4 $ and $G_0=\A_6 = \Sp_4(2)'< \SL_4(2)$. In this case, $ r=6 $ or $10$ ($ \A_6\cong\PSL_2(9) $). By Lemma~\ref{lem:pty-1}, $ k $ is a $ 2 $-power, and it is with  $3 \leqslant k < r$, so $k = 4$ or $8$, and the possible values of $b = vr/k$ are $20$, $24$, or $40$. However, a computation in $ \MA $ shows that $G$ has no subgroup of index $b$ with an orbit of length $k$ on $\calP$. Hence this case cannot occur.\vs

{\bf(d)} $ |V|=2^4 $, $G_0=\A_7<\A_8=\SL_4(2)$. In this case, $ G_0 $ has $ 2 $-transitive permutation representations of degree $r=7$ or $15$. Similar to {\bf(c)}, we have $k = 4$ or $8$, $b = vr/k=28$, $30$ or $60$, and by $ \MA $, $b \ne 28$ or $60$. For the remaining case, $(v,b,r,k,\lambda)=(16,30,15,8,7)$, and $ \MA $ yields $ H,K\leqslant G=\ZZ_2^4{:}\A_7 $ such that \[\mbox{$ H=\A_7,\, K=2^3{:}\GL_3(2) $ and $ H\cap K=\GL_3(2)$.}\] The triple $ (G,H,K) $ satisfies the conditions of Construction \ref{cons}, thereby giving rise to a $ G $-locally $ 2 $-transitive design $ \calD $ with above parameters. Since $ k=2^3 $, by Lemma \ref{lem:pty-1}, $ \calD $ is a subdesign of $ \AG_3(4,2) $. Since $ r=15 $, it follows that $\calD=\AG_3(4,2)$.\vs

{\bf(e)} $ |V|=3^6 $, $G_0 = \SL_2(13)<\Sp_6(3)$. In this case, $r=14$. Then $v \leqslant \lambda v < rk < r^2 \leqslant 14^2$, not possible.
\end{proof}

\section{The almost simple case}\label{sec-as}
In this section we study case {\bf (C)}. In other word, we study designs that satisfy the following hypothesis.

\begin{hypothesis}\label{hypo-2}
{\rm Let $\calD=(\calP,\calB,\calI)$ be a $G$-locally 2-homogeneous non-symmetric design such that $G$ is an almost simple 2-transitive group on $ \calP $ with socle $ T $. Let $ (\a,\b)\in\calP\times\calB $ be a flag. Note that $ r\geq 4 $ by Lemma \ref{k<=v/2}.}
\end{hypothesis}

The main result is as follows.

\begin{proposition}\label{AS}
Under {\rm Hypothesis~\ref{hypo-2}}, one of the following holds{\rm\,:}
\begin{itemize}
\item[\rm(1)] $ \calD=\PG_1(d-1,q) $, $ d>3 $, and $ \PSL_d(q)\lhd G\leqslant \PGammaL_d(q) $.\vskip0.03in
\item[\rm(2)] $\calD=U_H(q)$, $\PSU_3(q) \lhd G\leqslant \PGammaU_3(q)$ and $ \PGU_3(q)\leq G\l\phi^2\r $, where $\phi$ is the field automorphism.\vskip0.03in
\item[\rm(3)] $ \calD $ and $ (G,G_\a,G_{\a\b},G_\b) $ are listed in {\rm Table \ref{table_as}}.
\end{itemize}
\noindent Further, the case where $ G $ is not locally $ 2 $-transitive on $ \calD $ only occur in {\rm Row 1, Table \ref{table_as}}.

\begin{table}[htbp]
\footnotesize
\renewcommand\arraystretch{1.3}
\centering
\caption{~}\label{table_as}
\begin{tabular}{cccccccccc}
\toprule[1.5pt]
$v$ & $b$ & $r$ & $k$ & $\lambda$ & $G$ & $G_\a$ & $G_{\a\b}$ & $G_{\b}$ & {\rm Remark} \\
\midrule[1pt]
$8$ & $14$ & $7$ & $4$ & $3$ & $\PSL_2(7)$ & $7{:}3$ & $3$ & $\A_4$ & $\calD=\AG_2(3,2)$\\
$12$ & $22$ & $11$ & $6$ & $5$ & $\M_{11}$ & $\PSL_2(11)$ & $\A_5$ & $\A_6$ \\
$15$ & $35$ & $7$ & $3$ & $1$ & $\A_7$ & $\PSL_3(2)$ & $\S_4$ & $(\A_4 \times 3){:}2$ & $\calD=\PG_1(3,2)$\\
$21$ & $56$ & $16$ & $6$ & $4$ & $\PSL_3(4).\calO$ & $2^4{:}\A_5.\calO$ & $\A_5.\calO$ & $\A_6.\calO$ & $\calO=1,2$\\
$22$ & $77$ & $21$ & $6$ & $5$ & $\M_{22}.\calO$ & $\PSL_3(4).\calO$ & $2^4{:}\A_5.\calO$ & $2^4{:}\A_6.\calO$ & $\calO=1,2$\\
\bottomrule[1.5pt]
\end{tabular}
\end{table}
\end{proposition}

The proof of this result proceeds by analyzing almost simple 2-transitive groups, which were classified in \cite{HERING1985151}. It is divided into three parts\,: $ T=\PSL_n(q) $ (Lemma \ref{psl}), $ T=\PSU_3(q) $ (Lemma \ref{psu}), and $ T $ is one of the others (Lemma \ref{others}). 

\begin{lemma} \label{psl}
Under {\rm Hypothesis~$\ref{hypo-2}$}, if $T=\PSL_d(q)$, one of the following holds{\rm\,:}
\begin{itemize}
\item[\rm(1)] $\calD=\PG_1(d-1,q)$, $d>3$, and $ \PSL_d(q)\lhd G\leqslant \PGammaL_d(q) $. \vskip0.05in
\item[\rm(2)] $\calD$ is a unique design with parameters $(v,b,r,k,\lambda)=(21,56,16,6,4)$, and \[\mbox{$(G,G_\a,G_{\a \b},G_{\b})=(\PSL_3(4).\calO,2^4{:}\A_5.\calO,\A_5.\calO,\A_6.\calO)$, $\calO=1$ or $2$}.\]
\item[\rm(3)] $\calD$ is the unique design $ \AG_2(3,2) $ with $(v,b,r,k,\lambda)=(8,14,7,4,3)$, and $$(G,G_\a,G_{\a \b},G_{\b})=(\PSL_2(7),7{:}3,3,\A_4).$$
\end{itemize}
\noindent Further, $ G $ is locally $ 2 $-transitive on $ \calD $ in {\rm(1), (2)}.
\end{lemma}

\begin{proof}
By {\rm Hypothesis~$\ref{hypo-2}$}, $ G $ is 2-transitive on the set $ \calP $ of $ v $ points with $ \soc(G)=T $. Thus, by \cite[Thm.\,5.3]{as2trans}, either $v=\frac{q^d-1}{q-1}$, or one of the following two cases occurs.

The first case is $(T,v)=(\PSL_2(8), 28)$, where $ T $ does not acts $2$-transitively on $28$ points but $\PGammaL_2(8)$ does. By \cite{ATLAS}, $(G, G_\a)=(\PGammaL_2(8),\ZZ_9{:}\ZZ_{6}).$ Note that $ G_\a=\ZZ_9{:}\ZZ_{6} $ can only act $2$-homogeneously on $3$ points. Thus, $ k<r=3 $ and $ \calD $ is trivial.

The second case is $(T,T_\a,v)=(\PSL_2(11),\A_5,11)$. In this case, $ r=5 $ or $ 6 $. If $r=5$, then $ k=3$ or $4$, and $ bk=vr=5 \cdot 11  $, not possible. If $r=6$, then $k=3,4$ or $5$ and $10\lambda=\lambda(v-1)=r(k-1)=6(k-1)$, not possible.\vs

In the following, we consider the generic case $v=\frac{q^d-1}{q-1}$. In this case, $ G\leq\PGammaL_d(q) $ and $ \calP $ is regarded as the set of either $ 1 $-subspaces or hyperplanes of $ V=\mathbb{F}_q^d $. We remark that the two situations are treated similarly, and the resulting designs are isomorphic. Here, we assume the former holds, that is, $ \calP=\{\l v\r\mid v\in V\setminus\{\0\}\} $.\vs

Let $ K=\SL_d(q) $ be the preimage of $ T $ under the projective homomorphism, and $ L=K.\calO\leq\GammaL_d(q) $ be the preimage of $ G $, where $ \calO $ is generated by the diagonal and the field automorphism. Then $ K $, $ L $ act unfaithfully on $ \calP $, and, respectively, induces an automorphism group of $ \calD $, i.e., $ K^\calP=T $, $ L^\calP=G $. For convenience, we consider $ K,L $ rather than $ T,G $. Take a point $ \a\in\calP $. The stabilizer $ L_\a $ is equal to $ K_\a.\calO $, where $ K_\a $ is a parabolic subgroup as follows\,:
\[K_{\a}=\left<
\left[\begin{array}{cc}
a&0\\ B&C
\end{array}\right]
\Bigg|\ a=|C|^{-1},\ B\in\mathbb{F}_q^{d-1},\ C\in \GL_{d-1}(q) \right>\cong q^{d-1}{:}\GL_{d-1}(q).\] In particular, there is a minimal normal subgroup $\mathcal{N}_1\lhd K_\a  $, referred to as the {\it unipotent radical}, which is isomorphic to $ q^{d-1} $ and consists of matrices $ N_1 $ of form $$N_1=\left[	
\begin{array}{c|ccc}
	1 & 0 & \cdots & 0\\\hline
	n_2  & 1 & &  \\
	\vdots & & \ddots &  \\
	n_d    & & & 1\\
\end{array}
\right].$$ Note that $ \mathcal{N}_1 $ is also (minimal) normal in $ L_\a=K_a.\calO $. Thus, the projective image $\overline{\mathcal{N}}_1(\cong \mathcal{N}_1) \lhd  G_{\a}$. Since $G$ is locally 2-homogeneous on $ \calD $, it follows that either\vskip0.02in

{\bf(a)} $\mathcal{N}_1$ fixes all blocks in $\calD(\a)$\,; or\vskip0.02in

{\bf (b)} $\mathcal{N}_1$ is transitive on $\calD(\a)$.\vs\vs

\noindent\textbf{(a)} $\mathcal{N}_1$ fixes all blocks in $\calD(\a)$.

Let $ (\a,\b) $ be a flag with $ \a=\l e_1\r $, $\b=\{\langle e_1 \rangle, \cdots, \l e_s\r,\,\l v_{s+1}\r,\ldots,\langle v_k \rangle \}$. Let $V_{\b}\leq V$ be the subspace spanned by 1-subspaces lying in $ \b $. Thus, $ \b\subseteq\{\l u\r\,|\,u\in V_\b\} $. Without loss of generality, assume that $\{e_1, \cdots, e_s\}$ is a basis of $V_{\b}$. Expand it to a basis of $ V $ as $ \eta=\{e_1, \cdots, e_d\} $. Under the basis $ \eta $, vectors in $ V_\b $ are expressed as $$V_{\b}=\{[u_1,u_2,\cdots,u_s,0,\cdots,0]\,|\,u_i \in \FF_q \}.$$

Let $P_{ij}\in\GL_d(q) $ be the permutation matrix that interchanges $ i $-th and $ j $-th columns of identity matrix $I_d$. For $ 1\leq \ell \leq s $, let $\mathcal{N}_{\ell}=\mathcal{N}_1^{P_{1 \ell}}$. Then $ \mathcal{N}_{\ell} $ consists of matrices $ N_\ell $ of form
$$N_{\ell}=\left[	
\begin{array}{ccccccc}
	1& & & n_1 &  &  & \\
	&\ddots &  & \vdots &  &  & \\
	& & 1 & n_{\ell-1} &  &  & \\
	& & & 1   &  &  & \\
	& & & n_{\ell+1} &1 &  & \\
	& & & \vdots &  & \ddots & \\
	& & & n_d &  &  & 1\\
\end{array}\right].$$
Since $ \mathcal{N}_1\lhd H_{\a} $ and $ \mathcal{N}_1 $ fixes all blocks in $ \calD(\a) $, it follows that $\mathcal{N}_\ell=\mathcal{N}_1^{P_{1\ell}}$ is normal in $ H_{\a}^{P_{1\ell}}=H_{\l e_\ell\r} $. 
Hence $\ov{\mathcal{N}}_\ell= \ov{\mathcal{N}}_1^{\ov P_{1\ell}}$ is also normal in $ G_{\a}^{\ov P_{1\ell}}=G_{\l e_\ell\r} $ and $ \mathcal{N}_\ell $ fixes all blocks in $ \calD(\l e_\ell\r) $.
In particular, each of $\mathcal{N}_1,\ldots,\mathcal{N}_s$ fixes the block $ \b $, and hence $\b\supseteq\l e_1\r ^{\mathcal{N}_2 \mathcal{N}_3 \cdots \mathcal{N}_s\mathcal{N}_1}$. Note that for any $u=[u_1,u_2,\cdots,u_d] \in V $ and $ N_{\ell}\in\mathcal{N}_\ell $, $$u^{N_\ell}=uN_\ell=[u_1,\ldots,u_{\ell-1},\,n_1u_1+n_2 u_2+ \cdots +n_d u_d,\,u_{\ell+1},\cdots,u_d], \text{with $n_{\ell}=1$}.$$ 
Thus, \[e_1^{\mathcal{N}_2 \mathcal{N}_3 \cdots \mathcal{N}_s\mathcal{N}_1}=\{[x_1,x_2,\cdots,x_s,0,\cdots,0]\,|\,x_i \in \FF_q \} = V_{\b},\] and so $ \b\supseteq\{\l u\r\,|\,u\in V_\b\} $. Therefore, $ \b=\{\l u\r\,|\,u\in V_\b\} $.\vs 

Recall $ \dim V_\b=s $. Now, $ \calB=\b^G $, and each block is a set of 1-subspaces contained in some $ s $-subspace of $ V $. By Lemma \ref{c}, $ \calD $ is quasi-symmetric. It follows that any two $ s $-subspaces of $ V $ with non-trivial intersection intersect in a fixed dimension. Hence $s=2$ or $d-1$. For $s=d-1$, $\calD=\PG(d-1,q)$ is symmetric. Thus, $s=2$, $d>3$, and $\calD=\PG_1(d-1,q)$. By Example \ref{ex:Pro-G}, $ \calD $ is $ G $-locally $ 2 $-transitive for each $ G $ with $ \PSL_d(q)\lhd G\leqslant \PGammaL_d(q) $. Part~(1) is satisfied.\vs\vs

\noindent\textbf{(b)} $\mathcal{N}_1$ is transitive on $\calD(\a)$.

In this case, $G_{\a}$ is 2-homogeneous on $ \calD(\a) $, and meanwhile, $G_{\a}$ has an abelian minimal normal subgroup $\mathcal{N}_1\cong q^{d-1}$ that is transitive on $\calD(\a)$. Thus, $\mathcal{N}_1$ is regular on $\calD(\a)$, and $r=q^{d-1}$. We first prove that $ d\leq 3 $.\vs

Let $ q=p^f $ with $ p $ prime. Since $v-1=q \cdot \frac{q^{d-1}-1}{q-1}$, $r=q^{d-1}$ and $\lambda(v-1)=r(k-1)$, we have $\frac{q^{d-1}-1}{q-1}\div k-1$. The design $ \calD $ is non-symmetric, so $ k<r $, and thus $$k=h\cdot\frac{q^{d-1}-1}{q-1}+1, \mbox{\ where\ } 1\leq h<q-1. $$ Note that $k \equiv h+1\,(\mod q)$, so $q \nmid k$. Suppose $p^i\,||\, k$, where $0 \leqslant i < f$. Then $ \gcd(k,r)=\gcd(k,q^{d-1})=p^i $. Since $k\div vr$, it follows that $ k\div p^iv $, and so $$ k\ \Big|\  (p^iqk-p^ihv)=p^i\, (q\,(h\cdot\frac{q^{d-1}-1}{q-1}+1)-h\,(q\cdot\frac{q^{d-1}-1}{q-1}+1)   )=p^i(q-h).$$
Hence we have $$ hq^{d-2}<k<p^i q<q^2. $$ It follows that either $ d=2 $\,; or $ d=3 $, $ h<p^i $. Moreover, in the latter case, since $k \equiv h+1\,(\mod q)$ and $p^i\,||\, k$, we have $ p^i\div h+1 $, so $ h=p^i-1 $.\vs\vs

\noindent\textbf{(b.1)} We first treat the case $d=3$, and show that there is only a unique design on $ 21 $ points in this case. Now, $ h=p^i-1 $ and $ k=h(q+1)+1 $. Let $j=f-i>0$.\vs

If $i=j$, then $p^i=\sqrt{q}$, and $k=h(q+1)+1=(\sqrt{q}-1)(q+1)+1.$ Thus, $$\lambda=\frac{r(k-1)}{v-1}=\frac{q^2(\sqrt{q}-1)(q+1)}{q(q+1)}=q(\sqrt{q}-1).$$
By Lemma \ref{c}, the following number $ c $ is an integer\,: $$c=\frac{(k-1)(\lambda-1)}{r-1}+1=\frac{(\sqrt{q}-1)(q+1)(q\sqrt{q}-q-1)}{(q+1)(\sqrt{q}+1)(\sqrt{q}-1)}+1=\frac{q\sqrt{q}-q-1}{\sqrt{q}+1}+1.$$
It follows that $ \sqrt{q}+1\div 3 $, $q=4$ and $T=\PSL_3(4)$. The parameters are $(v,b,r,k,\lambda)=(21,56,16,6,4)$. By $ \MA $, for $ G=T.\calO $, where $\calO=1$ or $2$, there exist subgroups $ X,Y\leq G $, such that \[\mbox{$ X\cong\ZZ_{2}^4{:}\A_5.\calO $, $Y\cong\A_6.\calO $ and $ X\cap Y\cong \A_5.\calO $}.\] The triple $ (G,X,Y) $ satisfies the conditions of Construction \ref{cons}, thereby giving rise to a $ G $-locally $ 2 $-transitive design with above parameters. Part~(2) is satisfied.\vs

If $i \ne j$, then $v=q^2+q+1=p^{2i+2j}+p^{i+j}+1$,
and meanwhile, $$k=h(q+1)+1=p^i q -q +p^i=p^i(p^{i+j}-p^j+1).$$
This case cannot occur, since we have $k\div p^i v$, which leads to a contradiction\,:
\begin{align*}
\frac{k}{p^i}&=\gcd(\frac{k}{p^i},v)=\gcd(p^{i+j}-p^j+1,p^{2i+2j}+p^{i+j}+1)\\
&=\gcd(p^{i+j}-p^j+1,p^{2i+2j}+p^{j})=\gcd(p^{i+j}-p^j+1,p^{2i+j}+1)\\
&=\gcd(p^{i+j}-p^j+1,p^{i+j}-p^{i}+1)=\gcd(p^{i+j}-p^j+1,p^{j}-p^{i})<\frac{k}{p^i}.
\end{align*} 

\noindent\textbf{(b.2)} At last, we treat the case $d=2$ (so $ q\geq 4 $), and show that in this case $ \calD=\AG_2(3,2) $. Now, we have the following conditions\,: \[\mbox{$ v=q+1$,\, $ r=q=p^f $,\, $\lambda=k-1$,\, $ 3\leq k<q $, and $ k\div vr=(q+1)q $.}\]
Moreover, by Lemma \ref{c}, the number  $c=\frac{(k-1)(\lambda-1)}{r-1}+1=\frac{(k-1)(k-2)}{q-1}+1$ is an integer.\vs

Recall that $ K=\SL_2(q) $ acts unfaithfully on $ \calP $, which induces $K^\calP=T\leq\Aut(\calD)$. By Lemma~\ref{socflagtrans}, $T$ is flag-transitive on $\calD$, so is $ K $. Hence $K_{\a}\cong q{:}{(q-1)}$ is transitive on the set $
\calD(\a)$ of $ q $ blocks, where the stabilizer is $K_{\a\b}=\ZZ_{q-1}$. Note that there exists a maximal subgroup $H\leq K$ such that
$$K_{\a\b} < K_{\b} \leqslant H < K.$$ Further, $ |H| $ is divisible by $ |K_\b|=k|K_{\a\b}|=k(q-1) $.
We then analyze the maximal subgroups of $ K=\SL_2(q) $, which were determined in \cite[Table~8.1]{low-dimensional}, case by case.\vs

\noindent\textbf{(b.2.1)} $H \in \CC_1$. Then $H\cong q{:}(q-1)$ has $ K_{\a\b}=\ZZ_{q-1} $ as a maximal subgroup, so $K_{\b}=H$ and $\calD$ is symmetric, contrary to Hypothesis~$\ref{hypo-2}$.\vs

\noindent\textbf{(b.2.2)} $H \in \CC_2$. Then $|H|=2(q-1)$ is divisible by $k(q-1)$, not possible as $ k\geq 3 $.\vs

\noindent\textbf{(b.2.3)} $H \in \CC_3$. Then $|H|=2(q+1)$ is divisible by $k(q-1)$. Since $ 3\leq k<q $, it follows that $ k=3 $, $ q=5 $, and then $c=\frac{(3-1)(2-1)}{5-1}+1=\frac{3}{2}$, a contradiction.\vs

\noindent\textbf{(b.2.4)} $H \in \CC_5$. There are two subcases to treat.\vskip1pt

(i) $H=\SL_2(q_0)$, where $q=q_0^a$. Since $|H|=q_0(q_0^2-1)$ is divisible by $k(q_0^a-1)$, we have $a=2$ and $k\div q_0$. Thus, $ 0<c-1\leq\frac{(q_0-1)(q_0-2)}{q_0^2-1}<1 $, a contradiction.\vs

(ii) $H=\SL_2(q_0).2$, where $q=q_0^2$ is odd. Since $ |H|=2q_0(q_0^2-1) $ is divisible by $k(q_0^2-1)$, we have $k\div 2q_0$. If $k \leqslant q_0$, similar to the former case, $c$ is not an integer. Hence $k=2q_0$, and then
$$c-1=\frac{(2q_0-1)(2q_0-2)}{q_0^2-1}=\frac{4q_0-2}{q_0+1}=4-\frac{6}{q_0+1}.$$ It follows that $q_0=5$, and so $r=25$, $k=10$, $v=26$ and $b=vr/k=65$. Now, $T=\PSL_2(25)$ has a subgroup $T_{\b}$ of index $65$. By \cite{ATLAS}, we have $(G,G_{\a},G_{\b})=(\PSL_2(25),5^2{:}12,\S_5)$ or $(\PSigmaL_2(25),5^2{:}(4\times \S_3),\S_5 \times 2)$. Hence $G_{\a\b}=\ZZ_{12}$ or $\ZZ_{4}\times \S_3$, respectively. By $ \MA $, $\ZZ_{12}$ is not a subgroup of $\S_5$, and $\ZZ_{4}\times \S_3$ is not a subgroup of $\S_5\times \ZZ_{2}$. Hence this case cannot occur.\vs


\noindent\textbf{(b.2.5)} $H \in \CC_6$. There are two subcases to treat.\vskip1pt

(i) $H=2_-^{1+2}.\S_3$, where $q=p\equiv\pm 1$ (\mod $8$). In this case, we conclude that $ (q,k)=(7,4) $ is the unique pair satisfying that $3\leq k<q$, $k(q-1)\div 48$, and $q-1 \div(k-1)(k-2)$. Thus, $G=\PSL_2(7)$ or $\PGL_2(7)$.

If $G=\PGL_2(7)$, then $G_{\a}\cong \ZZ_7{:}\ZZ_6$, $G_{\a\b}\cong \ZZ_6$, and $|G_{\b}|=24$. By $ \MA $, any subgroup of $G$ of order $24$ is isomorphic to $\S_4$, so it has no subgroups isomorphic to $\ZZ_6$. Hence this case cannot occur. 

If $G=\PSL_2(7)$, then $G_{\a}\cong \ZZ_7{:}\ZZ_3$, $G_{\a\b}\cong \ZZ_3$, and $|G_{\b}|=12$. Any subgroup of $G$ of order $ 12 $ is isomorphic to $\A_4$ and intersects $ G_\a $ in $ \ZZ_3 $. Thus, the triple $(G,G_\a,G_{\b})=(\PSL_2(7),\ZZ_{7}{:}\ZZ_{3},\A_4)$ by Construction \ref{cons} induces a $ G $-locally 2-homogeneous design $ \calD $ with $(v,b,r,k,\lambda)=(8,14,7,4,3)$, which, by $ \MA $, is isomorphic to $ \AG_2(3,2) $. Part (3) is satisfied. In addition, note that $ G_\a $ is not $2$-transitive on $[G_{\a}{:}G_{\a \b}]$, so $ G $ is not locally 2-transitive on $ \calD $. \vs

(ii) $H=2_-^{1+2}{:}3$, where $q=p\equiv \pm 3,5,\pm 11, \pm 13, \pm 19$ (\mod $40$). In this case, we conclude that $ (q,k)=(5,3) $ is the unique pair satisfying that $3\leq k<q$, $k(q-1)\div 24$, and $q-1 \div(k-1)(k-2)$. Thus, $G=\PSL_2(5)$ or $\PGL_2(5)$.

If $G=\PGL_2(5)$, then $G_{\a}\cong \ZZ_5{:}\ZZ_4$, $G_{\a\b}\cong \ZZ_4$, and $|G_{\b}|=12$. By $ \MA $, any subgroup of $G$ of order $12$ is isomorphic to $\A_4$ or $\D_{12}$, so it has no subgroups isomorphic to $ \ZZ_4 $. Hence this case cannot occur. 

If $G=\PSL_2(5)$, then $G_{\a}\cong \ZZ_5{:}\ZZ_2$ is not $2$-homogeneous on $[G_{\a}{:}G_{\a \b}]$. Hence this case cannot occur, either.\vs

\noindent\textbf{(b.2.6)} $H \in \mathcal{S}$. There are two subcases to treat.\vskip1pt

(i) $H=2.\A_5$, where $q=p\equiv\pm 1$ (\mod $10$). In this case, we conclude that $ (q,k)=(11,6) $ is the unique pair satisfying that $3\leq k<q$, $k(q-1)\div 120$, and $q-1 \div(k-1)(k-2)$. Thus, $G=\PSL_2(11)$ or $\PGL_2(11)$.

If $G=\PGL_2(11)$, then $G_{\a}\cong \ZZ_{11}{:}\ZZ_{10}$, $G_{\a\b}\cong \ZZ_{10}$, and $|G_{\b}|=60$. By $ \MA $, any  subgroup of $G$ of order $60$  is isomorphic to $\A_5$, so it has no subgroups isomorphic to $\ZZ_{10}$. Hence this case cannot occur.

If $G=\PSL_2(11)$, then $G_{\a}\cong \ZZ_{11}{:}\ZZ_{5}$, $G_{\a\b}\cong \ZZ_{5}$, and $|G_{\b}|=30$. However, $G$ has no subgroups with order $30$. Hence this case cannot occur, either.\vs

(ii) $H=2.\A_5$, where $q=p^2$, $p\equiv\pm 3$ (\mod $10$). In this case, there exists no pair $(q,k)$ satisfying that $3\leq k<q$, $k(q-1)\div 120$, and $q-1 \div(k-1)(k-2)$.\vs

In conclusion, in the case $d=2$, we have that $ \calD=\AG_2(3,2) $.
\end{proof}

\begin{lemma} \label{psu}
Under {\rm Hypothesis~\ref{hypo-2}}, if $T=\PSU_3(q)$, then $\calD$ is the Hermitian Unitary $U_H(q)$, and $ G $ satisfies that $\PSU_3(q) \lhd G\leqslant \PGammaU_3(q)$ and $ \PGU_3(q)\leq G\l\phi^2\r $, where $\phi$ is the field automorphism. Further, $ G $ is locally $ 2 $-transitive on $ \calD $.
\end{lemma}

\begin{proof}
First, we have $ G\leq T.\Out(T)=\PGammaU_3(q) $. Let $ q=p^f $, where $ p $ is prime. Respectively, for the case $ \gcd(q+1,3)=1 $ or $ \gcd(q+1,3)=3 $ (thus, $ p\equiv -1\,(\mod 3) $ and $ f $ is odd), the outer automorphism group of $ T $ is
\begin{align*}
\Out(T)&=\l \d,\g,\phi\ |\ \d^{\gcd(q+1,3)}=\gamma^2=1,\,\phi^f=\gamma,\,\d^\g=\d^{-1},\,\d^\phi=\d^p\r\\
&\cong\ZZ_{2f} \mbox{\ or\ }\ZZ_{f}\times(\ZZ_3{:}\ZZ_2),
\end{align*} where $ \d $, $ \g $, $ \phi $ is the diagonal, the graph and the field automorphism.\vs

The group $G$ has a unique $2$-transitive action of degree $ q^3+1 $. It follows that, letting $V=\FF_{q^2}^3$ and $ (V,B) $ be a non-degenerate unitary space, the set $\calP$ is regarded as the set of all $q^3+1$ isotropic $1$-subspaces of $ V $. Note that $ T $ is transitive on $ \calP $.\vs

For $\a\in\calP$, the stabilizer $ \GU_3(q)_\a $ is given in Example~\ref{ex:psu3q}. Projectively, we have \[\PGU_3(q)_{\a}=Q{:}L, \mbox{\ where\ } Q=Q'.(Q/Q')=q.q^2,\, L=\ZZ_{q^2-1}.\] The subgroup $ Q $, known as the unipotent radical, is contained in $ T $ and is normalized by $ \phi $. Hence we have $$T_{\a}=(q.q^2){:}\ZZ_{\frac{q^2-1}{(3,q+1)}}, \mbox{\ and\ } \PGammaU_3(q)_{\a}=(q.q^2){:}(\ZZ_{q^2-1}{:}\ZZ_{2f}).$$ In particular, $ Q$, $Q'\lhd G_\a $. If $ Q' $ is transitive on $ \calD(\a) $, then $ r\div q $, and by $\lambda(v-1)=r(k-1)$, we have a contradiction\,: $$k=\frac{\lambda(v-1)}{r}+1=\frac{\lambda q^3}{r}+1 \geqslant \lambda q^2+1 >r.$$ Thus, $ Q' $ acts trivially on $ \calD(\a) $. Further, if $ Q $ acts trivially on $ \calD(\a) $, since $ T $ is flag-transitive on $ \calD $ (Lemma \ref{socflagtrans}), then $ r\div(q^2-1) $, $ (r,p)=1 $, and $r\div \lambda$ by $\lambda(v-1)=\lambda q^3=r(k-1)$, not possible. Thus, $ Q/Q'=q^2 $ is non-trivial on $ \calD(\a) $.\vs

View $ q^2 $ as the vector space $ \FF_p^{2f} $. Note that $ \ZZ_{\frac{q^2-1}{(3,q+1)}}=\ZZ_{q^2-1} $ or $ \ZZ_{\frac{q^2-1}{3}} $, which, respectively, has $ 1 $ or $ 3 $ orbits on the set of non-zero vectors of the space. In the former case, $ \ZZ_{q^2-1} $ is irreducible on $ \FF_p^{2f} $. 
Since $T$ is almost simple, we have $q>2$. In the latter case, assume that $X$ is a subspace fixed by $ \ZZ_{\frac{q^2-1}{3}} $. If $X$ has exactly $\frac{p^{2f}-1}{3}+1$ points, then $ p=2 $, $ f=1 $ and $ q=2 $, not in our case. Hence $X$ contains at least $2$ orbits of length $\frac{p^{2f}-1}{3}$.
However, $|X| \geqslant \frac{2(p^{2f}-1)}{3}+1>p^{2f-1}$.
Hence $\ZZ_{\frac{q^2-1}{3}}$ is irreducible on $\FF_p^{2f}$, too. Thus, $ Q/Q'= q^2 $ is a minimal normal subgroup of $ q^2{:}\ZZ_{\frac{q^2-1}{3}} $, and it follows that $ Q/Q' $ is an abelian normal subgroup of $ G_\a^{\calD(\a)} $, so it is regular on $ \calD(\a) $ and $ r=q^2 $.\vs

Now, $ r=q^2 $, $ k=\frac{\lambda(v-1)}{r}+1=\lambda q+1<r $, so $ \lambda<q $. By Lemma \ref{c}, $ (r-1)$ divides $(k-1)(\lambda-1) $, so \[q^2-1\ \Big|\ \lambda q(\lambda-1), \mbox{\ and then,\ } q^2-1\ \Big|\ \lambda(\lambda-1).\] Since $ \lambda(\lambda-1)<q(q-1)<q^2-1 $, we must have $ \lambda=1 $. The parameters are $$(v,b,r,k,\lambda)=(q^3+1,q^2(q^2-q+1),q^2,q+1,1).$$ According to \cite[Section 2]{Buekenhout2vk1}, a design with the above parameters can only be the Hermitian Unitary $U_H(q)$.\vs\vs

We then determine which group $ G $ lying between $\PSU_3(q)$ and $\PGammaU_3(q)$ is locally 2-homogeneous on $ \calD $. By Example \ref{ex:psu3q}, if $\PGU_3(q) \leq G$, then $G$ is locally $2$-transitive on $\calD$. We thus consider the case $\PGU_3(q)\nleq G$. This only happens when $\gcd(3,q+1)=3$ (thus, $ p\equiv -1\,(\mod 3) $ and $ f $ is odd), and $\PSU_3(q)<\PGU_3(q)$. 

The stabilizer of block is determined in Example \ref{ex:psu3q}. It follows that $$\GU_3(q)_{\b}=\GU(\b)\times\GU(\b^\bot)\cong\GU_2(q) \times \GU_1(q),$$ and then $\SU_3(q)_{\b}$ contains a subgroup $ \SU(\b) $ which induces the following group $$ \PSU(\b)=\PSU_2(q)\cong\PSL_2(q), $$ that acts 2-transitively on the set $ \calD(\b) $ of $ q+1 $ points. Since $ \SU_3(q) ^\calP=\PSU_3(q)\leq G $, we have that $ G_\b $ is always 2-transitive on $ \calD(\b) $.\vs

On the other hand, consider the local action of points. Since $ r=q^2 \not\equiv 3\,(\mod 4) $, the action is 2-homogeneous if and only if it is 2-transitive. It is shown above that $$ T_\a^{\calD(\a)}=T_\a/Q'= q^2{:}\ZZ_{\frac{q^2-1}{3}}, $$ where $ Q/Q'= q^2 $ is regular on $ \calD(\a) $, and the stabilizer $ (T_{\a}^{\calD(\a)})_\b=\ZZ_{\frac{q^2-1}{3}} $. View $ \calD(\a) $ as $ \FF_q^2 $. Let $ \b=\0 $ and $\FF_{q^2}^\times=\langle \omega \rangle \cong \ZZ_{q^2-1}$. Then $ (T_{\a}^{\calD(\a)})_\0$ has $3$ orbits on $ \FF_{q^2}^\times $ as
$$O_1=\{\omega^{1+3i}\in\FF_{q^2}^\times\},\, O_2=\{\omega^{2+3i}\in\FF_{q^2}^\times\},\, O_3=\{\omega^{3i}\in\FF_{q^2}^\times\}.$$
In particular, $ T_\a $ is not 2-transitive on $ \calD(\a) $. Now, let $ X=G/T\leq\Out(T) $. Then $ G_\a=T_\a.X $, and so $$G_\a^{\calD(\a)}=T_\a^{\calD(\a)}.X=q^2{:}(\ZZ_{\frac{q^2-1}{3}}.X).$$ Hence $ G_\a $ is $ 2 $-transitive on $ \calD(\a) $ if and only if $ X $ mixes the 3 orbits $ O_1,O_2,O_3 $ together. Recall that 
\begin{align*}
\Out(T)&=\l\phi^2\r\times\l\d,\phi^f\r\cong\ZZ_{f}\times(\ZZ_3{:}\ZZ_2),
\end{align*}
where $ \l\d\r=\l w\r/\l w^3\r $, $ \phi:\FF_{q^2}^\times\rightarrow\FF_{q^2}^\times$, $t\mapsto t^p $, and $ p\equiv -1\,(\mod 3)$, so\vs

$\delta:O_1 \rightarrow O_2  \rightarrow O_3 \rightarrow O_1$,\vskip1pt

$\phi:O_1 \rightarrow O_2  \rightarrow O_1,  O_3 \rightarrow O_3,$ and $\phi^2$ fixes all $O_1,O_2,O_3$.\vs

\noindent Now, let $ \tau:\Out(T) \rightarrow\Out(T)/\l\phi^2\r=\l\d,\phi^f\r $ be the projection mapping, where the image $  \l\d,\phi^f\r\cong \D_6$. Since $ X $ mixes $ O_1,O_2,O_3 $, it follows that $ X^\tau $ contains $ \l\d\r $. Hence, we have \[\mbox{$ \l\d\r\leq X\l\phi^2\r $, and $ \PGU_3(q)\leq G\l\phi^2\r $.}\]

In conclusion, we have that either $\PGU_3(q) \leq G\leqslant \PGammaU_3(q)$\,; or $\PGU_3(q)\nleq G $ and $ \PGU_3(q)\leq G\l\phi^2\r $. Further, in both cases, $ G $ is locally 2-transitive on $ \calD $. In particular, these two cases can be unified as \[\PGU_3(q)\leq G\l\phi^2\r.\]

At last, we remark that, only in the case $3\div f$ there exist groups $ G $ satisfying \[\mbox{$\PGU_3(q)\nleq G$ and $ \PGU_3(q)\leq G\l\phi^2\r $.}\] Note that $\PGU_3(q)\leq G\l\phi^2\r$ implies $ \l\d\r\leq X\l\phi^2\r $, and then $ X $ contains such an element $ x=\phi^{2i}\d $. If $ 3\nmid f $, the order $ o(\phi^{2i})=\frac{f}{(f,i)} $ is not divisible by $ 3 $. Thus, $ x^{o(\phi^{2i})}=\d $ or $ \d^{-1}$ is contained in $X$, and so $\PGU_3(q)\leq G$. Meanwhile, in the case $3\div f$, we may take $ \phi^{2i} $ of order divisible by $ 3 $, and take $ X $ to be the group of diagonal type $ X=\l\phi^{2i}\delta\r $. Now, $ \delta\notin X $ and $ \delta\in X\l\phi^2\r $. The condition is satisfied.
\end{proof}

\begin{lemma} \label{others}
Under {\rm Hypothesis~\ref{hypo-2}}, if $T\ne\PSL_d(q),$ $\PSU_3(q)$, then one of the cases in {\rm Row} $ 2 $, $ 3 $, $ 5 $ of {\rm Table \ref{table_as}} occurs. Further, $ G $ is locally $ 2 $-transitive on $ \calD $.
\end{lemma}

\begin{proof}
According to \cite[Thm.\,5.3]{as2trans}, either\vs

{\bf (a)} $T$ belongs to one of the following infinite families\,:
\[\mbox{$\A_v$,\, $\Sp_{2m}(2)$,\, $\Sz(q)$,\, $\Ree(q)$\,; or}\]

{\bf (b)} $T$ is one of the sporadic small groups\,:
\[\mbox{$\A_7$,\, $\HS$,\, $\M_{11}$,\, $\M_{12}$,\, $\M_{22}$,\, $\M_{23}$,\, $\M_{24}$,\, $\Co_3$.}\]

We first treat the infinite families.

{\bf(a.1)} Suppose $T = \A_v$. Then $G$ is $k$-homogeneous on $\calP$ and $ \calD $ is trivial.

{\bf(a.2)} Suppose $T = \Sp_{2m}(2)$, $m \geqslant 3$. The point stabilizer is $T_\a = \GO_{2m}^{\pm}(2)$. Note that neither $\rm GO_{2m}^{\pm}(2)$ ($m \geqslant 4$) nor $\rm GO_6^{-}(2)\cong\PSU_4(2).2$ admits a $2$-transitive permutation representation.
The only possible case is $(T, T_\a, v) = (\Sp_6(2), \GO_6^{+}(2), 36)$, where $ \GO_6^{+}(2)\cong\S_8 $. This case is excluded in {\bf(b.2)}.

{\bf(a.3)} Suppose $(T,v)=(\Sz(q),q^2+1)$, $ q=2^{2e+1}>2 $. The point stabilizer is $T_\a=\E_q^{1+1}.\ZZ_{q-1}$. Note that $ G\leq T.\calO $ and $ G_\a\leq T_\a.\calO $, where $ \calO=\Out(T) $ is solvable. Then $G_\a^{\calD(\a)}$ is of affine type with the socle being elementary abelian and isomorphic to a section of $\E_q$ or $\ZZ_{q-1}$. Thus, $r=|\calD(\a)|\leqslant q$. By $\lambda(v-1)=r(k-1)$, we have a contradiction $$k=\frac{\lambda(v-1)}{r}+1=\frac{\lambda q^2}{r}+1 \geqslant\lambda q+1 >r.$$
	
{\bf(a.4)} Suppose  $(T,v)=(\Ree(q),q^3+1)$, $q= 3^{2e+1}>3$. The point stabilizer is $T_\a=(\E_q^{1+1+1}){:}\ZZ_{q-1}$. Similar to {\bf(a.3)}, $G_\a^{\calD(\a)}$ is of affine type, $r=|\calD(\a)|\leqslant q$, and by $\lambda(v-1)=r(k-1)$, we have a contradiction
$k=\frac{\lambda q^3}{r}+1 \geqslant \lambda q^2+1 >r$.\vs\vs

We then treat the small candidates in {\bf(b)}. In the case $T = \Co_3$, the stabilizer $T_\a = \McL{:}2$ has no $2$-transitive permutation representation. The remaining cases are listed as follows, where the degree $v$ can be read off from \cite[Thm.\,5.3]{as2trans}, and further, the point stabilizer $T_\alpha$ and the possible overgroups $G$ are listed in~\cite{ATLAS} and \cite[Table\,7.4]{cameron-book}, respectively. Note that $G_{\a}=T_{\a}$ or $ T_{\a}.2 $ is an almost simple group that acts 2-transitively on the set $ \calD(\a) $ of $ r $ blocks. Again, the possible values of $r$ and the stabilizer $T_{\alpha\b}$ are determined by \cite[Thm.\,5.3]{as2trans} and the tables in~\cite{ATLAS}.

\[\mbox{\small
$\begin{array}{cccccccc} \hline
&T& T_\a & T_{\a\b} & v & r & G &  \\ \hline

1&\A_7 & \PSL_2(7) & \S_4,\ 7{:}3 & 15 & 7,8 & T \\

2&\Sp_6(2) & \S_8 & \S_7 & 36 & 8 & T \\

3&\HS & \PSU(3,5){:}2 & 5_+^{1+2}{:}8{:}2 & 176 & 126 & T \\  \hline

4&\M_{11} &\PSL_2(11) & \A_5, 11{:}5 & 12 & 11,12 &T  \\

5&\M_{11}&\M_{10} \cong \A_6.2 & 3^2{:}\Q_8 & 11 & 10 & T \\	

6&\M_{12} & \M_{11} & \M_{10},\ \PSL_2(11) & 12 & 11,12 &T   \\	

7&\M_{22} & \PSL_3(4) & 2^4{:}\A_5 & 22 & 21 & T,T.2 \\	

8&\M_{23} & \M_{22} & \PSL_3(4) & 23 & 22 & T\\	

9&\M_{24} & \M_{23} & \M_{22} & 24 & 23 & T \\
\hline
\end{array}$}\]\vs

{\bf(b.1)} Suppose $T=\A_7$, Then $ G=T $. If $r=8$, then $14\lambda=8(k-1)$, so $ k\geq 8= r $, not possible. Hence $r=7$, and then $14\lambda=7(k-1)$, so $2\lambda=k-1$. Since $k < r$ and $bk=vr=3 \cdot 5 \cdot 7$, we have $k=3$ or $5$, and, respectively, $\lambda=1$ or $2$. Further, by Lemma \ref{c}, $c=\frac{(k-1)(\lambda-1)}{r-1}+1$ is an integer. Thus, $(k,\lambda)=(3,1)$, and hence $(v,b,r,k,\lambda)=(15,35,7,3,1)$. By $ \MA $, there exists $ K\leqslant T $ such that  \[\mbox{$K\cong(\A_4 \times 3){:}2$, and $T_\a\cap K=\S_4$.}\] By Construction \ref{cons}, the triple $ (T,T_\a,K)=(\A_7,\PSL_2(7),(\A_4\times 3){:}2) $ induces a $ T $-locally 2-transitive design $ \calD $ with above parameters. In particular, $\lambda=1$, so by \cite[Section 2]{Buekenhout2vk1}, we have $\calD=\PG_1(3,2)$. This gives Row 3 of Table \ref{table_as}.\vs

{\bf(b.2)} Suppose $T=\Sp_6(2)$. Then $35\lambda=8(k-1)$, so $k\geqslant 36=v$, not possible.\vs
	
{\bf(b.3)} Suppose $T=\HS$. Then $175\lambda=126(k-1)$, so $25\div k-1$. Since $bk=vr=2^5\cdot 3^2 \cdot 7 \cdot 11$, we have $k= r=126$ and $\calD$ is symmetric, not in our case.\vs

{\bf(b.4)} Suppose $T\in\{\M_{11},\M_{12},\M_{22},\M_{23},\M_{24}\}$. If $r=v$, then $\lambda(v-1)=v(k-1)$, so $v-1 \div k-1$ and $v=k$, not possible. Hence $r=v-1$, and so $\lambda=k-1$. By Lemma \ref{c}, $c=\frac{(k-1)(\lambda-1)}{r-1}+1=\frac{(k-1)(k-2)}{r-1}+1$ is an integer. Thus, $k$ satisfies that
\[3 \leqslant k <r,\ k \div vr,\mbox{\ and\ }r-1 \div (k-1)(k-2).\]
\noindent It follows that the possible cases are\,: $$(T,v,k)\in\{(\M_{11},12,6),\, (\M_{12}, 12,6),\, (\M_{22}, 22,6),\, (\M_{24}, 24,12)\}.$$

{\bf(b.4.1)} Suppose $(T,v,k)=(\M_{11}, 12,6)$. Then $ G=T $, and the parameters $(v,b,r,k,\lambda)=(12,22,11,6,5)$. By $ \MA $, there exists $ K\leqslant T $ such that \[\mbox{$K\cong \A_6$, and $T_\a\cap K=\A_5$.}\] By Construction \ref{cons}, the triple $ (T,T_\a,K)=(\M_{11},\PSL_2(11),\A_6) $ induces a $ T $-locally 2-transitive design $ \calD $ with above parameters. In particular, $b=2v-2$, so by Lemma \ref{quasi_sym_b=2v-2}, $\calD$ is the Hadamard $3$-$(12,6,2)$ design. This gives Row 2 of Table \ref{table_as}.\vs

{\bf(b.4.2)} Suppose $(T,v,k)=(\M_{12}, 12,6)$. Then $(v,b,r,k,\lambda)=(12,22,11,6,5)$.
Note that $\M_{12}$ is 5-transitive on the set $ \calP $ of $ 12 $ points. It follows that $ \calD $ is a quasi-symmetric $4$-design, so by Lemma \ref{quasi_sym_4design}, $b =\frac{1}{2}v(v-1)$, a contradiction.\vs

{\bf(b.4.3)} Suppose $(T,v,k)=(\M_{22}, 22,6)$. Then $(v,b,r,k,\lambda)=(22,77,21,6,5)$. By \cite{ATLAS}, $ G=T.\calO $, where $ \calO=1 $ or $ 2 $, and further, if $ G=T.2 $, then $ G_\a=\PSL_3(4).2\cong\PSigmaL_3(4) $. By $ \MA $, there exists $ K\leqslant G $ such that \[\mbox{$K\cong 2^4{:}\A_6.\calO$, and $G_\a\cap K=2^4{:}\A_5.\calO$.}\] By Construction \ref{cons}, the triple $ (G,G_\a,K)=(\M_{22}.\calO,\PSL_3(4).\calO,2^4{:}\A_6.\calO) $ induces a $ G $-locally 2-transitive design $ \calD $ with above parameters. Since $\M_{22}$ is 3-transitive on the set $ \calP $ of $ 22 $ points, we have $\calD$ is a flag-transitive $3$-design. In particular, $\lambda_3 {22-2 \choose 3-2}/{6-2 \choose 3-2}=\lambda_2=5$, so $ \lambda_3=1 $. 
Such designs, 
were studied in \cite[Thm.\,3]{Homogeneous_designs}, by which $\calD$ is the 
Mathieu-Witt $3$-$(22,6,1)$ design. 
This gives Row 5 of Table \ref{table_as}.\vs

{\bf(b.4.4)} Suppose $(T,v,k)=(\M_{24}, 24,12)$. Then $(v,b,r,k,\lambda)=(24,46,23,12,11)$.
Note that $\M_{24}$ is $5$-transitive on the set $ \calP $ of $ 24 $ points. It follows that $ \calD $ is a quasi-symmetric $4$-design, so by Lemma \ref{quasi_sym_4design}, $b =\frac{1}{2}v(v-1)$, a contradiction.
\end{proof}

We conclude with the \noindent{\bf Proof of Theorem \ref{thm:2-trans-Design}\,:}

Let $\calD=(\calP,\calB,\calI)$ be a $ G $-locally $ 2 $-homogeneous design. By Lemma \ref{loc-2-trans}, the group $ G $ is 2-homogeneous on the set $ \calP $ of points. Then, by Wagner-Kantor's Lemma, there are the following three cases to consider\,:\vskip0.02in
{\bf(A)} $G$ is $2$-transitive on $ \calP $ and $ \calD $ is symmetric. This is treated in Section \ref{sec-Kantor} (Proposition  \ref{sym}). It yields parts {\bf(1.1)}, {\bf(1.2)}, as well as the cases in Table \ref{table_sym}.\vskip0.02in

{\bf(B)} $G$ is $2$-homogeneous on $ \calP $ of affine type, and either $G\leqslant \AGammaL_1(v)$\,; or $G\not\leqslant \AGammaL_1(v)$, $ \calD $ is non-symmetric. This is treated in Section \ref{sec-affine} (Proposition \ref{lem:1-dim} and Proposition \ref{Affine}). The former yields Row 3 of Table~\ref{table_nsym} and Row 1, 2 of Table~\ref{table_not2trans}. The latter yields part {\bf(1.3)} and Row 4 of Table~\ref{table_nsym}.\vskip0.02in

{\bf(C)} $G$ is $2$-transitive on $ \calP $ of almost simple type and $ \calD $ is non-symmetric. This is treated in Section \ref{sec-as} (Proposition \ref{AS}). It yields parts {\bf(1.4)}, {\bf(1.5)}, as well as Row 3 of Table~\ref{table_not2trans} and Row 1, 2, 5, 6 of Table~\ref{table_nsym}.\vskip0.02in

At last, we note that all the sporadic cases in Tables \ref{table_sym}\,-\,\ref{table_not2trans} hold by $ \MA $. \qed
	
\bibliographystyle{siam}
\bibliography{ref}

\section*{Declarations}
The author(s) declares that there is no any financial/personal relationship with other people/organizations not mentioned that can inappropriately influence the work.

\end{document}